\documentclass[11pt]{article}
\usepackage[utf8]{inputenc}
\usepackage{amsmath,amssymb,amsthm,amsfonts}
\usepackage{tikz-cd}
\usepackage{graphicx}
\usepackage{subcaption}
\usepackage{eucal}
\usepackage{enumitem}
\usepackage{geometry}
\usepackage{color}
\usepackage{authblk}
\numberwithin{equation}{section}

\theoremstyle{plain}
\newtheorem{theorem}{Theorem}[section]
\newtheorem{proposition}[theorem]{Proposition}
\newtheorem{lemma}[theorem]{Lemma}
\newtheorem{corollary}[theorem]{Corollary}
\theoremstyle{remark}
\newtheorem{remark}[theorem]{Remark}
\newtheorem{examples}[theorem]{Examples}

\newcommand{\gr}{\mathbf{gr}}
\newcommand{\grd}{\gr(D)}
\newcommand{\opn}{\operatorname}
\newcommand{\SK}{\opn{SK_1}}
\newcommand{\TK}{\opn{TK_1}}
\newcommand{\Dmu}{D^{(\mu)}}

\newcommand{\Emu}{E^{(\mu)}}
\newcommand{\Eone}{E^{(1)}}
\newcommand{\Nrd}{\opn{Nrd}}
\newcommand{\zz}{\mathbb Z}
\newcommand{\nn}{\mathbb N}
\newcommand{\qq}{\mathbb Q}
\newcommand{\ff}{\mathbb F}
\newcommand{\ind}{\text{ind}}
\newcommand{\degg}{\opn{deg}}
\newcommand{\charr}{\opn{char}}
\newcommand{\ov}{\overline}

\newcommand{\wt}{\widetilde}
\newcommand{\upmu}{{(\mu)}}
\newcommand{\upone}{{(1)}}
\newcommand{\bH}{\mathbf{H}}
\newcommand{\divr}{\opn{Div}_r}
\newcommand{\scT}{\mathcal T}
\newcommand{\scB}{\mathcal B}
\newcommand{\scS}{\mathcal S}
\newcommand{\scO}{\mathcal O}

\title{The group $\mathrm{TK}_1$ of graded and valued division algebras}
\author[1]{Huynh Viet Khanh}
\author[1]{Nguyen Duc Anh Khoa}
\author[2]{Adrian R. Wadsworth}
\affil[1]{Department of Mathematics and Informatics, Ho Chi Minh City University of Education, Ho Chi Minh City, Vietnam\\\texttt{khanhhv@hcmue.edu.vn}; \texttt{anhkhoa27092002@gmail.com}}
\affil[2]{Department of Mathematics, University of California San Diego, La Jolla, CA 92093-0112, USA\\\texttt{arwadsworth@ucsd.edu\qquad\qquad\qquad\qquad\qquad\qquad\qquad\qquad\qquad\qquad\qquad  \ \!\phantom{.}}}
\date{}

\begin{document}
\maketitle

\begin{abstract}
For a division algebra $D$, let $K_1(D) = D^*/[D^*, D^*]$ and let $\operatorname{TK}_1(D)$ be the torsion subgroup of 
the abelian group $K_1(D)$.
 We study this torsion group for graded and valued division algebras, in parallel with the known theory 
 of $\operatorname{SK}_1$. For a graded division algebra $E$ finite-dimensional over its center, we give 
 exact sequences describing $\operatorname{TK}_1(E)$ in terms of $E_0$, the grade group~$\Gamma_E$, 
 and the conjugation action of $E^*$ on $E_0$. These yield explicit formulas for $\operatorname{TK}_1(E)$ in 
 the unramified, totally ramified, and semiramified cases.
 
For a tame valued division algebra $D$ over its Henselian-valued center $K$, 
we identify the obstruction group $\mathbf H$ to a congruence theorem for $\TK(D)$.  We show that if the 
residue field~$\overline K$ of the valuation on $K$ has characteristic $p > 0$, 
then $\mathbf H \cong\mu_K[p]$, the $p$-primary component of 
the group $\mu_K$ of roots of unity in $K$; but  if $\operatorname{char}(\overline K)=0$, then $\mathbf H=1$. 
We further 
prove a short exact sequence 
$$
        1\,\longrightarrow \,\mathbf H\,
        \longrightarrow \,\operatorname{TK}_1(D)\,
        \longrightarrow\, \operatorname{TK}_1(\gr(D))\,
        \longrightarrow \,1,
$$
where $\gr(D)$ is the associated graded division algebra determined by the valuation on $D$
obtained from the valuation on $K$.

We also prove a stability theorem for a graded division algebra $E$ with quotient division ring~$q(E)$, i.e.,
$$
        \operatorname{TK}_1(E)\,\cong \,\operatorname{TK}_1(q(E)),
$$
together with a new proof of the corresponding stability theorem for $\operatorname{SK}_1$. As applications, we 
obtain graded analogues of Motiee's primary decomposition and scalar-extension results for torsion Whitehead groups.
\end{abstract}

\section{Introduction}

Let $D$ be a division algebra finite-dimensional over its center $K$. The reduced Whitehead group 
$\opn{SK}_1(D)=D^{(1)}/D'$, where $D^{(1)}=\{d\in D^*\mid \opn{Nrd}_D(d) = 1\}$ and $D' = [D^*,D^*]$, 
has played an important role in the structure theory of division algebras and in applications to algebraic groups. 
A fundamental tool in its explicit computation in specific cases is the use of valuations. If $K$ is Henselian, 
the valuation $v$ on $K$ extends uniquely to $D$, and the associated graded ring $\grd$ determined by $v$ is  
a graded division algebra. The passage from $D$ to $\grd$ 
has been shown to be  particularly useful because many computations become simpler in the graded setting.

The  approach via associated graded rings was developed systematically by Hazrat and Wadsworth in their study 
in \cite{hazrat-wadsworth-2011} of $\SK$ of graded and valued division algebras. They proved, among other things, 
graded analogues of Ershov's exact sequences in \cite{ershov-1982} for $\SK(D)$ where $D$ is a division algebra over 
a Henselian field, and 
deduced explicit formulas for $\operatorname{SK}_1$ in the unramified, totally ramified, and semiramified cases.  
They further proved  the isomorphism 
$\opn{SK}_1(D)\cong\opn{SK}_1(\grd)$ for tame division algebras $D$ over Henselian fields,
thereby obtaining formulas for $\SK(D)$ from those for $\SK(\gr(D))$. Additionally, they proved  the stability isomorphism 
$\SK(E)\cong \SK(q(E))$ for a graded division algebra $E$, where $q(E)$ is the central quotient division ring of $E$. 
These results are also presented and further developed in \cite[Ch.~11]{tignol-wadsworth-2015}.

The purpose of this paper is to study the corresponding questions for the torsion subgroup $\operatorname{TK}_1(D)=\tau(K_1(D))=\tau(D^*/D')$ of the Whitehead group $K_1(D) = D^*/D'$. The congruence theorem for 
$\operatorname{SK}_1$ says that $(1+M_D)\cap D^{(1)}\subseteq D'$, where $M_D$ is the maximal ideal for the 
valuation on $D$, and it is a key tool for proving the isomorphism $\SK(D) \cong \SK(\gr(D))$. But, 
unlike $\operatorname{SK}_1(D)$, the group 
$\operatorname{TK}_1(D)$ is not controlled solely by reduced norm one elements, and there is no 
full analogue for $\TK(D)$ to the congruence theorem for $\SK(D)$.   
In~\cite{khanh-khoa-2025}, Khanh and Khoa isolated the  obstruction to a congruence theorem for $\TK$. 
Namely, if $D^{(\mu)}$ is the inverse image of $\TK(D)$ in~$D^*$, the obstruction group is
$$
        \mathbf{H}\,=((1+M_D)\cap D^{(\mu)})\big/((1+M_D)\cap D').
$$
The congruence theorem for $\operatorname{TK}_1(D)$ holds precisely when $\mathbf{H}=1$.
In Prop.~\ref{prop:obstruction-norm}  below we give a complete description of this obstruction group. Namely,
$$
        \mathbf{H}\,\cong (1+M_K)\cap\mu_K.
$$
We deduce from Lemma~\ref{lem:roots-of-unity-s} that  if the residue field $\overline K$ has characteristic $p>0$, then $\mathbf{H}\cong\mu_K[p]$, the $p$-primary component of the group of roots of unity in $K$,  while if 
$\operatorname{char}(\overline K)=0$, then $\mathbf{H}=1$. Remarkably, the obstruction group thus depends only 
on the roots of unity in the  
center $K$, not on the particular tame division algebra $D$ over $K$. These results  complete and slightly correct 
the computations for $\mathbf{H}$  given in \cite{khanh-khoa-2025}.

Having determined $\mathbf{H}$, our next step is to relate $\operatorname{TK}_1(D)$ to $\TK(\gr(D))$. We prove 
in Th.~\ref{thm:valued-tk1-exact-sequence} 
the short exact sequence
\begin{equation}
\label{SESfor$TK_1$}
        1\longrightarrow \mathbf{H}
        \longrightarrow \operatorname{TK}_1(D)
        \longrightarrow \operatorname{TK}_1(\grd)
        \longrightarrow 1.
\end{equation}
This is the $\operatorname{TK}_1$-analogue of  \cite[Th.~4.8]{hazrat-wadsworth-2011}  saying that 
$\SK(D) \cong\SK(\grd)$.
It shows that the only obstruction to having an isomorphism 
$\operatorname{TK}_1(D)\cong\operatorname{TK}_1(\grd)$ is 
given by $\mathbf{H}$. 

Before considering the obstruction group and applications to $\TK(D)$ we first establish for a graded division ring $E$
analogous formulas for $\TK(E)$  to formulas  given in \cite{hazrat-wadsworth-2011} for $\SK(E)$.  In 
Th.~\ref{thm:graded-tk1-diagram} we give a general diagram relating $\TK(E)$ to data from  $E_0$ and $\Gamma_E$
and the conjugation action of $E^*$ on $E_0$; 
this diagram is 
 analogous to the one in \cite[Th.~3.4]{hazrat-wadsworth-2011}.
  In Cor.~\ref{cor:compressed-tk1-sequence} we further compress some of the data in this big diagram into a single 
  short exact sequence.  
 We analyze the diagram and its compressed version to obtain in 
 Th.~\ref{theorem:graded-tk1-cases}  formulas for 
 $\TK(E)$ if $E$ is unramified, totally ramified, or semiramified.  After (\ref{SESfor$TK_1$}) is proved, we deduce
 in Cor.~\ref{cor:TK_1(D)cases}
 formulas for $\TK(D)$ where $D$ is a tame division algebra over its Henselian center from the corresponding 
 formulas for 
 $\TK(\gr(D))$.

We also prove  stability for $\TK$ and a new proof of stability for $\SK$. For a graded division algebra $E$ with 
quotient division algebra $Q=q(E)$, we first prove
in Prop.~\ref{prop:cokernel-torsion-free}  that the cokernel of the natural map $K_1(E)\to K_1(Q)$, namely
  $Q^*/E^*Q'$, is torsion-free. From this we  deduce in Th.~\ref{thm:stability-sk1-tk1} the stability 
  theorems
$$
        \operatorname{TK}_1(E)\cong \operatorname{TK}_1(Q)
        \quad \text{and}\quad
        \operatorname{SK}_1(E)\cong \operatorname{SK}_1(Q).
$$

Finally, we record some consequences of the stability theorem for $\TK$. 
By applying results of Motiee in \cite{motiee-2013} on 
$\TK$ of any 
finite-dimensional  division algebra and invoking the stability theorem, we obtain corresponding results for 
$\TK$ of a graded division algebra.   This includes in Th.~\ref{prdecompTK_1E}
 a graded primary decomposition theorem and in Prop.~\ref{prop:prime-to-index-injectivity} an injectivity result under 
 finite graded field extensions of degree 
 relatively prime to the index. 
 
The paper is organized as follows. In Section~2 we fix notation and conventions. In Section~3 we prove exact sequences 
for $\operatorname{TK}_1$ of a graded division algebra $E$ and derive from them explicit formulas for $\TK(E)$ in the 
unramified, totally ramified, and semiramified cases. In Section~4 we determine the obstruction group $\mathbf{H}$. In 
Section~5 we prove the exact sequence (\ref{SESfor$TK_1$}) and use it to deduce formulas for $\TK(D)$ from those 
for $\TK(\gr(D))$.
In Section~6 we prove the stability theorem for $\operatorname{SK}_1$ and $\operatorname{TK}_1$. 
In Section~7 we give applications to primary decompositions, finite graded field extensions, and tensor products.

\section{Notation and conventions and basic facts}\label{sec:notation}
For an abelian group $A$, let $\tau(A)$ denote the torsion subgroup of $A$. If $K$ is a field, let $\mu_K$ be 
the group of roots of unity in $K$; thus $\mu_K = \tau(K^*)$.  For a positive integer $m$, let 
$\mu_K(m)=\{\zeta\in\mu_K\mid \zeta^m=1\}$; and for a prime number $p$ let 
$\mu_K[p]$ denote the $p$-primary component of $\mu_K$.

We next set some notation and recall basic properties of graded division rings.  For omitted proofs, see
\cite{hazrat-wadsworth-2011} or \cite{tignol-wadsworth-2015}.  Let $E$ be a graded division ring graded by 
a torsion-free abelian group~$\Gamma_E$.  That is, $E$ is a graded ring 
$E =\bigoplus_{\gamma\in \Gamma_E} E_\gamma$ such that each nonzero homogeneous element $e_\gamma$
in any $E_\gamma$ is a unit of $E$.  It follows (as $\Gamma_E$ can be totally ordered) that the group of 
units~$E^*$ of $E$ consists entirely of its nonzero homogeneous elements.  
For $e\in E^*$, the degree $\opn{deg}(e)$ is the $\gamma\in \Gamma_E$ such that $e\in E_\gamma$.  Clearly
the degree map $\deg\colon E^* \to \Gamma_E$ is  a group epimorphism with kernel $E_0^*$.  Thus, $\opn{deg}$
induces an isomorphism
\begin{equation}
\label{degiso}
E^*/E_0^* \, \cong\, \Gamma_E.
\end{equation}
Denote by $E'$
the multiplicative commutator subgroup $[E^*, E^*]$  of $E^*$.  Since commutators of homogeneous elements must 
have  degree $0$, it follows that  $E' \subseteq E_0^*$.  Clearly, $E_0$ is a division ring. We 
always assume that each $E_\gamma$ is nonzero. So, $E_\gamma$~is a one-dimensional left and right
$E_0$-vector space.  

Let $T = Z(E)$, the center of $E$. Then (as $\Gamma_E$ can be totally ordered), $T$ is a graded subring of 
$E$, i.e., $T = \bigoplus_{\delta\in \Gamma_T} T_\delta$, where $\Gamma_T \subseteq \Gamma _E$ and 
each $T_\delta \subseteq E_\delta$.  Also, $T$ is a graded field, i.e., a commutative graded division ring.
Since $E$ is a graded $T$-module, it is  free as a $T$-module, with rank denoted~$[E:T]$.   Then, 
$[E:T] = [E_0:T_0] \cdot |\Gamma_E:\Gamma_T|$.  
We call $E$ a {\it{graded division algebra}} when (in addition to $\Gamma_E$ torsion-free) $[E:T] < \infty$.
 We say that the graded division algebra $E$ is {\it{unramified}} if $\Gamma_E = \Gamma_T$ 
 (so $[E:T] = [E_0: T_0]$);  $E$ is {\it{totally ramified}} if $|\Gamma_E:\Gamma_T| = [E:T]$ (so $E_0 = T_0$); and 
$E$ is {\it{semiramified}} if $[E_0: T_0] = |\Gamma_E:\Gamma_T|$ and $E_0$ is commutative. 

The graded
field $T$ is an integral domain;  so, $T$ has a quotient field, denoted $q(T)$. 
The quotient division ring $q(E)$ of $E$ is its ring of central quotients $q(E) = E \otimes _T q(T)$.  This 
$q(E)$ is a division ring with center $Z(q(E)) = q(T)$, and $[q(E): q(T)] = [E:T] < \infty$. 
The Schur index of $E$ is defined to be $\text{ind}(E) = \sqrt{[E:T]\,}$; so $\text{ind}(E) = \text{ind} (q(E)) \in \nn$. 

The reduced norm map $\Nrd_{q(E)}\colon q(E) \to q(T)$ sends $E$ into $T$. We denote by $\Nrd_E$
the restriction of $\Nrd_{q(E)}$ to $E$.  This map is compatible with the grading  on $E$ in that for any
$\gamma\in\Gamma_E$, we have  $\Nrd_E(E_\gamma) \subseteq T_{m\gamma}$, where 
$m = \text{ind}(E)$; moreover, if $\Nrd_E(b) \in T_\delta \setminus\{0\}$, then $\delta = m\rho$ for some 
$\rho \in \Gamma_E$ and $b\in E_\rho$.  In particular, $\Nrd_E(E_0) \subseteq T_0$, and if $\Nrd_E(c) \in T_0$,
then $c\in E_0$.  The field~$T_0$ is often a proper subfield
of $Z(E_0)$.  But $Z(E_0)$ is always abelian Galois over $T_0$.

The Whitehead group of $E$ is $K_1(E) = E^*/E'$. Let $\TK(E)$ be the  torsion subgroup of $K_1(E)$,
i.e., $\TK(E) =\tau(E^*/E')$. Let $\SK(E) = \Eone/E'$, 
where $\Eone= \{ a \in E^*\mid \Nrd_E(a) = 1\}$.  Then, $\SK(E) \subseteq \TK(E)$, as $\SK(E)$ has 
$\text{ind}(E)$-torsion by 
\cite[Prop.~3.3(ii)]{hazrat-wadsworth-2011}.  Let $\Emu$ be the subgroup of $E^*$ containing $E'$ such that
$\Emu / E' = \TK(E)$; equivalently, ${\Emu = \{ a\in E^*\!\mid\! \Nrd_E(a) \in \mu_{T_0}\}}$.  We have the inclusions:
$$
[E_0^*, E_0^*] \, \subseteq \, [E_0^*, E^*] \, \subseteq \,E' \subseteq  \, \Eone \, \subseteq \,
\Emu\, \subseteq \, E_0^*.
$$
Note also that $E^*/\Emu$ is torsion-free, as  $E^*/\Emu \cong (E^*/E')\big/(\Emu/E') =
(E^*/E')\big/ \tau(E^*/E')$.

\smallskip

In the valued setting, let $D$ be a division algebra finite-dimensional over its center $K$. Let $v$ be a valuation on $D$.
We will assume throughout that the restriction of $v$ to $K$ is Henselian. 
Let~$V_D$~be the valuation ring of $v$ on $D$; let $M_D$ be the maximal ideal of $V_D$ and $\overline D=V_D/M_D$ 
the residue division algebra; and let $\Gamma_D$ be the value group, which is a torsion-free abelian group, written 
additively.  The group of units  of $V_D$ is $V_D^* = \{a\in D^*\mid v(a) = 0\}$.
 Likewise let  $V_K$, $M_K$, 
 $\overline K$, $\Gamma_K$, and $V_K^*$ be the corresponding objects for $K$. 
 We will always assume that $D$ is {\it{tame}} (= tamely ramified over $K$).  This means that 
 $[D:K] = [\ov D: \ov K]|\Gamma_D:\Gamma_K|$, $Z(\ov D)$ is separable (hence abelian Galois) over $\ov K$, and 
 either $\charr(\ov K) = 0$ or $\charr(\ov K) = p>0$ and $p \nmid \ind(D)\big /\big (\ind(\ov D) [Z(\ov D):\ov K]\big )$.
 $D$ is said to be {\it {unramified}} if $\Gamma_D = \Gamma_K$ (so $[\ov D: \ov K] = [D:K])$; $D$ is {\it{totally ramified}}
 if $|\Gamma_D:\Gamma_K| = [D:K]$ (so $\ov D = \ov K$); and $D$ is  {\it{semiramified}} if 
 $[\ov D:\ov K] = |\Gamma_D:\Gamma_K| = \opn{ind}(D)$ and $\ov D$ is commutative.
 
 Let $\gr(D)$ be the associated graded ring of $D$ with respect to $v$.  This is defined as follows:
 For $\gamma \in \Gamma_D$, let 
 $$
 D^{\ge \gamma} \, = \, \{ d\in D^*\mid v(d) \ge \gamma\} \cup \{0\} \quad \text {and} \quad
 D^{> \gamma} \, = \, \{ d\in D^*\mid v(d) > \gamma\} \cup \{0\} \, \subseteq \, D^{\ge \gamma}. 
 $$
 Then set 
 $$
 \gr(D) = \underset{\gamma\in \Gamma_D}{\textstyle{\bigoplus}} \gr(D)_\gamma, \quad \text{where}\ \quad \ 
 \gr(D)_\gamma\, = \, D^{\ge\gamma} /D^{> \gamma}.
 $$
 For $a\in D^*$, let $\wt a$ denote the image $a + D^{>v(a)} $ of $a$ in $\gr(D)_{v(a)}$.
 With the well-defined multiplication on homogeneous elements of $\gr(D)$ given 
 by $\wt a \cdot\wt b = \wt{\,ab\,}$, this $\gr(D)$ is a graded division ring.  Note that 
 $\Gamma_{\gr(D)} = \Gamma_D$ and $\gr(D)_0 = \ov D$. Furthermore,
 $\gr(K)$, defined analogously, is the center of~$\gr(D)$ and 
 $[\gr(D):\gr(K)] = [D: K]$. A further property (which we will not use),  is that $\gr(D)$~determines $D$ up to 
 isomorphism.  Note that the homomorphism $D^* \to \gr(D)^*$ given by $a \mapsto \wt a$
 has kernel $1+ M_D$. 
 
 Let $\Nrd_D\colon D \to K$ be the reduced norm map.
 Put $D'=[D^*,D^*]$, ${D^{(1)}=\{d\in D^*\mid \Nrd_D(d)=1\}}$, and $\operatorname{TK}_1(D)=\tau(D^*/D')$. 
 Let $\Dmu$ be the inverse image of $\operatorname{TK}_1(D)$ in $D^*$, 
 so that $D'\subseteq \Dmu$ and $\Dmu/D'=\operatorname{TK}_1(D)$. Note that 
$\Dmu = \{d\in D^*\mid \operatorname{Nrd}_D(d)\in \mu_K\}$.  The reduced norm is compatible with the 
valuation, in that $v(\Nrd_D(a)) = \opn{ind}(D)\,v(a)$ for all $a\in D^*$.  By \cite[Cor.~4.4]{hazrat-wadsworth-2011}, there 
is also compatibility with $\Nrd_{\gr(D)}$, in that for any $a\in D^*$,
\begin{equation}
\label{eq:Nrdcompatibility}
\Nrd_{\gr(D)}(\wt a) \, = \, \wt{\Nrd_D(a)}.
\end{equation}  
There is also the useful identity, by \cite[Cor.~4.7]{hazrat-wadsworth-2011},
\begin{equation}
\label{eq:Nrd(1+M_D)}
\Nrd_D(1+M_D) \, = \, 1+M_K.
\end{equation}

\section{The group $\operatorname{TK}_1$ of graded division algebras}\label{sec:graded-tk1}

We first work in the graded setting.  We give the analogue, for $\operatorname{TK}_1$, of the graded calculations 
in \cite{hazrat-wadsworth-2011} for $\operatorname{SK}_1$.  These formulas will be applied later in  
Cor~\ref{cor:TK_1(D)cases} easily yielding corresponding formulas for $\TK(D)$
of a tame division algebra $D$ over a Henselian field.  The diagram for the next theorem is the analogue for 
$\TK$ of the diagram for $\SK$ of a graded division algebra given in \cite[Th.~3.4]{hazrat-wadsworth-2011}, which was in 
turn based on the diagram for $\SK$ of a tame division algebra over a Henselian field given in \cite{ershov-1982}.

Throughout this section, let $E$ be a graded division algebra with center the graded field $T$,
 as described in Section~2.  The basic relationship  
between $\TK(E)$ and $\SK(E)$ is as follows:

\begin{proposition}
\label{prop:TKvsSKforE}
For any graded division algebra $E$ with center $T$, there is a short exact sequence:
\begin{equation}
\label{SESforTK(E)}
1 \longrightarrow  \SK(E) \longrightarrow \TK(E) \longrightarrow \Nrd_{E}(E_0^*)\cap \mu_{T_0} \longrightarrow 1
\end{equation}
\end{proposition}

\begin{proof}
Recall the inclusions noted in Section~2 above: 
$E' \subseteq E^{(1)} \subseteq \Emu \subseteq E_0^*$; hence, ${\SK(E) \subseteq \TK(E)}$.  
Recall also that $E^*/\Emu$ is torsion-free.
Now, take any 
$a \in E^{(\mu)}$.  Then, $a\in E_0^*$ and $a^m \in E'$ for some $m\in \nn$. Therefore, 
$1 = \Nrd_E(a^m) = (\Nrd_E(a))^m$.
Also, $\Nrd_E(a) \in T_0$ as $a \in E_0$.  This shows that 
$\Nrd_E(E^{(\mu)}) \subseteq \Nrd_{E}(E_0^*)\cap \mu_{T_0}$.
For the reverse inclusion, take any $b \in E_0^*$, such that $\Nrd_E(b) \in \mu_{T_0}$.  Then, for some 
positive integer $k$,
$ 1= \Nrd_E(b)^k = \Nrd_E(b^k)$, i.e., $b^k \in E^{(1)} \subseteq E^{(\mu)}$.  Therefore, $b\in \Emu$, as
$E^*/\Emu$ is torsion-free.  
Thus, $\Nrd_E$ maps $\Emu$ onto  $\Nrd_{E}(E_0^*)\cap \mu_{T_0}$.  This surjectivity implies exactness at the
right term of (\ref{SESforTK(E)}), and exactness at the center and left terms is clear from the inclusions noted at the 
beginning of the proof. 
\end{proof}

The proposition shows that $\TK(E)/\SK(E)$ is a locally cyclic abelian group (as is every subgroup of $\mu_{T_0}$).  
However, this factor group is often very difficult to calculate in practice.  Therefore, we proceed to analyze 
$\TK(E)$ in parallel with the known theory for $\SK(E)$  given in \cite{hazrat-wadsworth-2011} and also in 
\cite[Ch.~11]{tignol-wadsworth-2015}.

Before stating the main theorem for $\TK(E)$ we need some further notation.  For a graded division algebra
$E$, let $\mathcal I_E=[E_0^*,E^*]$.  
Let $T = Z(E)$ and $Z_0 = Z(E_0)$, which is abelian Galois over~$T_0$.  Let $G =\opn{Gal}(Z_0/T_0)$, and let
$N_{Z_0/T_0}\colon Z_0 \to T_0$ be the field norm.  Let 
$$A \,=\,  \Nrd_{E_0}(E_0^*) \,\subseteq \,Z_0^*.
$$
  Then, as each
$\sigma\in G$ extends to an automorphism of $E_0$ (conjugation by some element of $E^*$), $\sigma$~restricts 
to an automorphism of $A$.  Let 
$$
        I_{G}(A)=\,\langle \,a\,\sigma(a)^{-1}\!\mid a\in A,\ \sigma\in G\rangle \quad \text{and} \quad
        A^{(\mu)}=\, \{a\in A\mid N_{Z_0/T_0}(a)\in\mu_{T_0}\}.
$$
Since for any $a\in A$ and $\sigma\in G$, we have $N_{Z_0/T_0}(\sigma(a)) = N_{Z_0/T_0}(a)$, it 
follows that ${I_G(A) \subseteq A^{(\mu)}}$.  Define  
$$
        \mathcal H_{\mu}(G,A)\,=\,A^{(\mu)}/I_{G}(A).
$$ 
The interplay between $\Nrd_E$ on $E_0$ and $\Nrd_{E_0}$ will be essential here.  Recall from 
\cite[Prop.~3.2(iv)]{hazrat-wadsworth-2011} that for any $a\in E_0$, 
\begin{equation}
 \label{E_0Nrdformula}
\Nrd_E(a) \, =  \,\big(N_{Z_0/T_0}(\Nrd_{E_0}(a))\big)^{d_E}\quad \text {where} \quad 
d_E \,=\, \text{ind}(E)\big/\big(\text{ind}(E_0)[Z_0:T_0]\big) \in \nn.
\end{equation}

\begin{theorem}
\label{thm:graded-tk1-diagram}
Let $  E = \bigoplus_{\gamma \in \Gamma} E_\gamma  $ be a graded division ring finite-dimensional over its 
center $  T  $. 
Let  $\widetilde{N}=N_{Z_0/T_0}\circ \mathrm{Nrd}_{E_0}:E_0^*\to T_0^*$. 
Then, the two rows and the column of the following diagram are exact.  
\begin{equation*}
\begin{tikzcd}[
  column sep=1.5em,
  row sep=2.5em,
  ampersand replacement=\&,
  cells={nodes={inner sep=6pt}}
]
\& \& 1 \arrow[d] \\
\& \& \ker(\widetilde{N})/\mathcal I_E \arrow[d, "\scriptstyle\iota"] \\
\big(\mathcal I_E\cap E_0^{(1)}\big)\big/E_0' \  \arrow[r]\ 
\& \ \mathrm{SK}_1(E_0) \ \arrow[r, "\scriptstyle\jmath"]\ 
\& \ E^{(\mu)}/\mathcal I_E\ 
    \arrow[r, "\scriptstyle\kappa_\mu"]
    \arrow[d, equal]
\& \ \mathcal H_{\mu}(G,A)\  \arrow[r]
\& \ 1 \\
\Gamma_E/\Gamma_T \wedge \Gamma_E/\Gamma_T\ 
  \arrow[rr, "\scriptstyle\alpha"]
\& \&
\ E^{(\mu)}/\mathcal I_E\ 
  \arrow[r, "\scriptstyle\overline{\varphi}"]
  \arrow[d, "\scriptstyle\overline{N}"]
\& \ \mathrm{TK}_1(E)\  \arrow[r]
\& \ 1 \\
\& \& \mu_{T_0} \cap \widetilde{N}(E_0^*) \arrow[d] \\
\& \& 1
\end{tikzcd}
\end{equation*}
\begin{center}
{$\text{Diagram\,1}$}
\end{center}
\end{theorem}

The mappings  $\jmath$, $\kappa_\mu$, $\iota$, $\ov N$, $\alpha$,   and $\ov{\varphi}$ in the diagram will be defined 
during the proof.

\begin{proof}
We first prove the exactness of the upper row, as in the proof of the
upper row in \cite[Th.~3.4]{hazrat-wadsworth-2011}.  The inclusions
$E_0^{(1)}\subseteq E^{(1)}\subseteq E^{(\mu)}$  and $E_0' \subseteq \mathcal I_E$ yield a natural map 
$$
\jmath\colon\operatorname{SK}_1(E_0)\longrightarrow E^{(\mu)}/\mathcal I_E.
$$
Clearly, $\opn{ker} j = \big(\mathcal I_E\cap E_0^{(1)}\big)\big/E_0' $.

We now define $\kappa_\mu$.  If $a\in E^{(\mu)}$, then $a\in E_0^*$, so
$\operatorname{Nrd}_{E_0}(a)\in A$.  Moreover 
$N_{Z_0/T_0}(\operatorname{Nrd}_{E_0}(a)) = \widetilde N(a)\in\mu_{T_0}$, 
so $\operatorname{Nrd}_{E_0}(a)\in A^{(\mu)}$.  Since
$\operatorname{Nrd}_{E_0}(\mathcal I_E)=I_{G}(A)$, the reduced norm
induces a well-defined homomorphism
$$
        \kappa_\mu:E^{(\mu)}/\mathcal I_E
        \longrightarrow
        \mathcal H_\mu(G,A),\qquad
        a\mathcal I_E\longmapsto
        \operatorname{Nrd}_{E_0}(a)I_{G}(A).
$$

We show that map $\kappa_\mu$ is onto:  Let $u\in A^{(\mu)}$, and  choose any 
$a\in E_0^*$ such that $\Nrd_{E_0}(a)=u$.  Since
$u\in A^{(\mu)}$, we have $\widetilde N(a)=N_{Z_0/T_0}(u)\in\mu_{T_0}$. Hence,
as  $\operatorname{Nrd}_E(a)=\widetilde N(a)^{d_E}$ by (\ref{E_0Nrdformula}), 
we have $\Nrd_E(a)\in \mu_{T_0}$.  Choose
$m\geq1$ such that $\operatorname{Nrd}_E(a)^m=1$.  Then
$a^m\in E^{(1)}$.  Since $E^{(1)}/E'$ is torsion, there is $r\geq1$
with $a^{mr}\in E'$.  Thus $aE'$ is torsion in $E^*/E'$, so
$a\in E^{(\mu)}$.  Hence, $uI_{G}(A)$ lies in the image of
$\kappa_\mu$.

Finally, let $a\mathcal I_E\in\ker\kappa_\mu$.  Then $\Nrd_{E_0}(a)\in I_{G}(A) = \Nrd_{E_0}(\mathcal I_E)$. 
Choose $c\in\mathcal I_E$ such that $\Nrd_{E_0}(c)=\Nrd_{E_0}(a)$. Then $ac^{-1}\in E_0^{(1)}$, and 
$a\mathcal I_E=ac^{-1}\mathcal I_E$. Thus every element of $\ker\kappa_\mu$ lies in the image of $\jmath$.
The reverse inclusion is clear.  Hence $\ker\kappa_\mu=\operatorname{im}\jmath$. 
This proves the exactness of the upper row.

We next prove that the  column of the diagram is exact.  The inclusion
$\mathcal I_E\subseteq\ker\widetilde N$ is exactly the one used in
\cite[Th.~3.4]{hazrat-wadsworth-2011}.  Also, if $a\in\ker\widetilde N$,
then $\operatorname{Nrd}_E(a)=\widetilde N(a)^{d_E}=1$ by (\ref{E_0Nrdformula}); hence
$\ker\widetilde N\subseteq E^{(1)}\subseteq E^{(\mu)}$.  Thus,
 $$
 \iota\colon \ker\widetilde N/\mathcal I_E\,\longrightarrow \,E^{(\mu)}/\mathcal I_E
 $$ 
 is well-defined and injective.

We next define $\overline N$.  Let $g\in E^{(\mu)}$.  Then $g\in E_0^*$,
and there is $m\geq 1$ such that $g^m\in E'$.  Therefore
$$
        1=\operatorname{Nrd}_E(g^m)
        =\operatorname{Nrd}_E(g)^m
        =\widetilde N(g)^{d_E m},
$$
again by (\ref{E_0Nrdformula}).  Hence
$\widetilde N(g)\in\mu_{T_0}$, and so $\widetilde N(E^{(\mu)}) \subseteq \mu_{T_0}\cap\widetilde N(E_0^*)$. Since 
$\mathcal I_E\subseteq\ker\widetilde N$, we get a well-defined induced map
\begin{equation*}
        \overline N\colon E^{(\mu)}/\mathcal I_E
        \longrightarrow
        \mu_{T_0}\cap\widetilde N(E_0^*),
        \quad \text{given by} \quad
        a\mathcal I_E\longmapsto\widetilde N(a).
\end{equation*}

We show this map is onto:  Let $b\in\mu_{T_0}\cap\widetilde N(E_0^*)$.  Choose
$a\in E_0^*$ with $\widetilde N(a)=b$.  Since $b$ is torsion, $b^k=1$
for some $k\geq1$.  Hence $\widetilde N(a^k)=1$, so
$a^k\in\ker\widetilde N\subseteq E^{(1)}\subseteq E^{(\mu)}$.  
Since $E^*/\Emu$ is torsion-free, it follows that 
 $a\in E^{(\mu)}$, and $\overline N(a\mathcal I_E)=b$.

Finally, $\ker\overline N = (E^{(\mu)}\cap\ker\widetilde N)/\mathcal I_E = \ker\widetilde N/\mathcal I_E$, because 
$\ker\widetilde N\subseteq E^{(\mu)}$.  This proves the exactness of the column.

It remains to check the lower row.  We first describe the map $\alpha$.
For this, note that  the isomorphism $\ov\degg^{\,-1}\colon \Gamma_E \to E^*/ E_0^*$ maps 
$\Gamma_T$ onto $T^*E_0^*/E_0^*$; so it induces an isomorphism $\beta \colon \Gamma_E/\Gamma_T 
\to  E^*/T^*E_0^*$.  Let $B = E^*/T^*E_0^*$.  By \cite[Lemma~3.5]{hazrat-wadsworth-2011}, there is a well-defined 
epimorphism $\psi\colon B\wedge B \to E'/\mathcal I_E$ given by 
${(x\,T^*E_0^*)\wedge (y\,T^* E_0^*) \mapsto [x,y]\,\mathcal I_E}$, 
for any nonzero $x, y \in E^*$.  The map~$\alpha$ is the composition $\psi \circ(\beta\wedge\beta)$, i.e., 
$$
\begin{aligned}
&\alpha\colon \Gamma_E/\Gamma_T\wedge \Gamma_E /\Gamma_T \to E^{(\mu)} /\mathcal I_E \quad 
\text {is the well-defined  map with image $E'/\mathcal I_E$ given by} \\
&\quad
(\gamma +\Gamma_T)\wedge (\delta + \Gamma_T)
\mapsto [x,y] \mathcal I_E,
\quad \text{for any $\gamma, \delta \in \Gamma_E$
and any nonzero $x\in E_\gamma$ and $y \in E_\delta$}.
\end{aligned}
$$
The map 
$$
\overline\varphi:E^{(\mu)}/\mathcal I_E\,
        \longrightarrow \,E^{(\mu)}/E'\,=\,\operatorname{TK}_1(E)
        $$
         is the quotient map, so it is surjective  with 
         $\ker\overline\varphi=E'/\mathcal I_E =\opn{im}(\alpha)$.  Thus, the lower row of the diagram is exact.
\end{proof}

\begin{remark}
\label{rem:torsion-norm-row}
The upper row of Diagram $1$ is  the torsion-norm analogue of the corresponding
\cite{hazrat-wadsworth-2011} row for $\operatorname{SK}_1$.  Namely, 
$\widehat H^{-1}(G,A) = \ker(N_{Z_0/T_0}:A\to T_0^*)/I_{G}(A)$ is replaced here
by $\mathcal H_\mu(G,A) = A^{(\mu)}/I_{G}(A)$. Thus $N_{Z_0/T_0}$ gives an exact sequence
$$
        1\longrightarrow
        \widehat H^{-1}(G,A)
        \longrightarrow
        \mathcal H_\mu(G,A)
        \longrightarrow
        \mu_{T_0}\cap\widetilde N(E_0^*)
        \longrightarrow 1.
$$
So the extra information in passing from $\operatorname{SK}_1$ to
$\operatorname{TK}_1$ is precisely the torsion part of the norm image.  If
$\alpha$ is trivial, this upper row of the diagram gives a direct description of
$\operatorname{TK}_1(E)$.
\end{remark}

\medskip

We now take the quotient of the middle column in Diagram~1 by the image of
$\alpha$.  Since ${\operatorname{im}\alpha=E'/\mathcal I_E}$, this quotient
is precisely $\operatorname{TK}_1(E)=E^{(\mu)}/E'$.  The resulting exact
sequence gives a more compact form of the lower portion of Diagram 1:

\begin{corollary}
\label{cor:compressed-tk1-sequence}
\textit{Let $\mathcal C_E=\widetilde N(E')$. Equivalently, $\mathcal C_E = (\overline N\circ\alpha)
        \bigl(\Gamma_E/\Gamma_T\wedge\Gamma_E/\Gamma_T\bigr)$. Then there is a short exact sequence}
\begin{equation}
\label{eq:compressed-tk1-sequence}
        1\longrightarrow
        \frac{\ker\widetilde N}{E'\cap\ker\widetilde N}\,
        \longrightarrow
        \operatorname{TK}_1(E)\,
        \overset{\Psi_\mu}{\longrightarrow} \,
        \frac{\mu_{T_0}\cap\widetilde N(E_0^*)}{\mathcal C_E}
        \longrightarrow 1 .
\end{equation}
\textit{The map $\Psi_\mu$ is given by $\Psi_\mu(a\,E')=\widetilde N(a)\,\mathcal C_E$, where $a\in E^{(\mu)}$.}

\end{corollary}
\begin{proof}
Consider the following commutative diagram:
$$
\begin{tikzcd}[column sep=large,row sep=large]
1 \arrow[r]
& E'\cap \ker(\wt N) \arrow[r] \arrow[d,hook]
& E' \arrow[r,"\wt N"] \arrow[d,hook]
& \mathcal C_E \arrow[r] \arrow[d,hook]
& 1 \\
1 \arrow[r]
& \ker(\wt N)\arrow[r]
& \Emu  \arrow[r,"\wt N"]
& \mu_{T_0} \cap \wt N(E_0^*)  \arrow[r]
& 1 
\end{tikzcd}
$$
In this  diagram, the columns are all inclusions.  The upper row is clearly exact. The proof for 
Th.~\ref{thm:graded-tk1-diagram} 
of exactness of the column in Diagram 1 showed that $\wt N$ maps $\Emu$ onto  $\mu_{T_0} \cap \wt N(E_0^*)$.
Therefore, the lower row of this diagram is exact.  By the $5$-Lemma, there is a short exact sequence of 
cokernels of the columns in the diagram.  This short exact sequence is (\ref{eq:compressed-tk1-sequence}).  
\end{proof}

\begin{remark}
\label{rem:compressed-sk1-comparison}
The same construction applied to the \cite{hazrat-wadsworth-2011} diagram for
$\operatorname{SK}_1(E)$ gives the analogous exact sequence
\begin{equation}\label{eq:compressed-sk1-sequence}
        1\longrightarrow
        \frac{\ker\widetilde N}{E'\cap\ker\widetilde N}\,
        \longrightarrow
        \operatorname{SK}_1(E)
        \,\overset{\Psi_1}{\longrightarrow}\,
        \frac{\mu_{T_0}(d_E)\cap\widetilde N(E_0^*)}{\mathcal C_E}
        \longrightarrow 1,
\end{equation}
where $\mathcal C_E=\widetilde N(E')$, and $\Psi_1$ defined by $\Psi_1(aE')=\widetilde N(a)\mathcal C_E$, with 
$a\in E^{(1)}$.
Thus the compressed sequences \eqref{eq:compressed-sk1-sequence} for $\operatorname{SK}_1(E)$ and
\eqref{eq:compressed-tk1-sequence} for $\operatorname{TK}_1(E)$ have the same kernel and the same 
denominator
$\mathcal C_E$ on the right.  The only change is that $\mu_{T_0}(d_E)\cap\widetilde N(E_0^*)$ is replaced by $\mu_{T_0}\cap\widetilde N(E_0^*)$. Thus the passage from $\operatorname{SK}_1(E)$ to 
$\operatorname{TK}_1(E)$
amounts, in this formulation, to replacing the $d_E$-torsion in the norm
image by the full torsion in the norm image; the contribution from the grade
group commutator pairing is still measured by $\mathcal C_E=\widetilde N(E')$.
\end{remark}

\begin{theorem}
\label{theorem:graded-tk1-cases}
Let $E$ be a graded division ring finite-dimensional over its center $T$. Then the following assertions hold:
\begin{enumerate}[font=\normalfont]
\item[(i)] If $E$ is unramified, then $\mathrm{TK}_1(E)\cong\mathrm{TK}_1(E_0)$.

\item[(ii)] If $E$ is totally ramified, then $\mathrm{TK}_1(E)\cong \mu_{T_0}/\mu_{T_0}(e)$, where 
$e=\exp(\Gamma_E/\Gamma_T)$ and $\mu_{T_0}(e)=\{\zeta\in T_0^*\mid \zeta^e=1\}$.

\item[(iii)] If $E$ is semiramified, then $E_0$ is a field which is Galois over $T_0$. In this case $Z(E_0)=E_0$, so
$$
        G\,=\,\operatorname{Gal}(E_0/T_0)\cong\,\Gamma_E/\Gamma_T,
$$
and
$$
        \mathcal H_\mu(G,E_0^*)
        =
        \{a\in E_0^*\mid N_{E_0/T_0}(a)\in\mu_{T_0}\}/I_{G}(E_0^*).
$$
Then the following sequence is exact:
$$
        G\wedge G
        \longrightarrow
        \mathcal H_\mu(G,E_0^*)
        \longrightarrow
        \mathrm{TK}_1(E)
        \longrightarrow 1.
$$
Moreover, there is a natural short exact sequence
\begin{equation}
\label{SESsemiram}
        1\longrightarrow
        \mathrm{SK}_1(E)
        \longrightarrow
        \mathrm{TK}_1(E)
        \longrightarrow
        \mu_{T_0}\cap N_{E_0/T_0}(E_0^*)
        \longrightarrow 1.
\end{equation}

\item[(iv)] If $E$ has maximal graded subfields $L$ and $K$ which are respectively unramified and totally ramified 
over $T$, then $E$ is semiramified, so $G=\operatorname{Gal}(E_0/T_0)$, and 
$\mathrm{TK}_1(E)\cong \mathcal H_\mu(G,E_0^*)$. In particular, there is a short exact sequence
\begin{equation}
\label{SESH^-1forTK_1}
        1\longrightarrow
        \widehat H^{-1}(G,E_0^*)
        \longrightarrow
        \mathrm{TK}_1(E)
        \longrightarrow
        \mu_{T_0}\cap N_{E_0/T_0}(E_0^*)
        \longrightarrow 1.
\end{equation}
\item[(v)]  If $\Gamma_T \cong \zz$, then $d_E =1$, and  there are exact sequences 
\begin{equation}
\label{ES:TK_1(E)Gamma_cyclic}
 1 \, \longrightarrow \, \big([E_0^*, E^*]\cap E_0^{(1)}\big)\big/E_0' \, \longrightarrow \,\SK(E_0) \, 
 \longrightarrow \, \TK(E) \, \longrightarrow \, H_{\mu}(G,A) \, \longrightarrow \,1
\end{equation} 
and
\begin{equation}
\label{SES:TK_1(E)Gamma_cyclic}
1 \, \longrightarrow \,  \ker(\wt N)/ [E_0^*, E^*] \, \longrightarrow \,  \TK(E)\, \longrightarrow \, 
 \mu_{T_0} \cap \widetilde{N}(E_0^*) \, \longrightarrow \,  1.  
\end{equation}
\end{enumerate}

\end{theorem}

\begin{proof}
We first prove \textup{(i)}.
The isomorphism $E^*/E_0^* \cong \Gamma_E$  given by $\ov{\degg}$ maps 
$T^*E_0^*$ onto $\Gamma_T$; hence $T^*E_0^*/E_0^* \cong \Gamma_T$.    Assuming $E$ is unramified, we have 
$\Gamma_E = \Gamma_T$, hence $E^* = T^* E_0^*$.  As $T^*$ is central in
$E^*$, it follows that $E' = E_0'$.  Consider the natural exact sequence:
$$
        1\longrightarrow E_0^*/E_0'\longrightarrow T^*E_0^*/E_0'\longrightarrow T^*E_0^*/E_0^* \longrightarrow 1
$$
Here, the middle term is isomorphic to $E^*/E'$ and the right term is isomorphic to $\Gamma_T$, which is 
torsion-free.  Hence, we have isomorphism of the left and middle torsion subgroups, i.e., 
$$
\TK(E_0) =  \tau(E_0^*/E_0') \cong  \tau(T^*E_0^*/E_0') \cong \tau(E^*/E') =  \TK(E).
$$

We next prove \textup{(ii)} using the compressed exact sequence \eqref{eq:compressed-tk1-sequence}. If $E$ is totally 
ramified, then $E_0=T_0$, so $\widetilde N=\operatorname{id}_{T_0^*}$ and $\ker\widetilde N=1$. Hence the 
compressed exact sequence gives
$$
        \mathrm{TK}_1(E)\cong \mu_{T_0}\big / \widetilde N(E')=\mu_{T_0}/E'.
$$
By the description of the commutator subgroup of a totally ramified graded division algebra, we have $E'=\mu_{T_0}(e)$, 
where $e=\exp(\Gamma_E/\Gamma_T)$; see \cite[Prop.~2.1]{hwang-wadsworth-99} and 
compare \cite[Cor.~3.6(ii)]{hazrat-wadsworth-2011}. Consequently $\mathrm{TK}_1(E)\cong\mu_{T_0}/\mu_{T_0}(e)$.

We now prove \textup{(iii)}. Suppose that $E$ is semiramified. That is, 
$[E_0:T_0] = |\Gamma_E:\Gamma_T| = \opn{ind}(E)$, and $E_0$ is a field.  
So, $Z_0 = Z(E_0) = E_0$, hence $E_0$ is Galois over~$T_0$, and $G=\opn{Gal}(E_0/T_0)$, and the epimorphism
$\Gamma_E/\Gamma_T \to G$ must be an isomorphism by comparing cardinalities.  Note also that 
$$
d_E \, = \, \opn{ind}(E) \big/(\opn{ind}(E_0) [Z_0:T_0])  \,=\, |\Gamma_E:\Gamma_T|\big/[E_0:T_0]  \,=\, 1.
$$
We saw in the proof of exactness of the upper row of Diagram 1  that $I_{G}(A) = \Nrd_{E_0}(\mathcal I_E)$.  Here, 
$A = E_0^*$ and  $\Nrd_{E_0}=\opn{id}_{E_0}$; hence $\mathcal I_E=I_{G}(E_0^*)$.
 Since $E_0$ is a field, $\mathrm{SK}_1(E_0)=1$; so the upper row of Diagram~1 identifies 
 $E^{(\mu)}/I_{G}(E_0^*)\cong \mathcal H_\mu(G,E_0^*)$.
Under this identification, the lower row of Diagram~1 becomes the exact sequence
\begin{equation}
\label{SESforH_mu}
        G\wedge G
        \longrightarrow
        \mathcal H_\mu(G,E_0^*)
        \longrightarrow
        \mathrm{TK}_1(E)
        \longrightarrow 1.
\end{equation}
Here the first map sends 
$(\gamma+\Gamma_T)\wedge(\delta+\Gamma_T) \longmapsto [x_\gamma,x_\delta]\,I_{G}(E_0^*)$ for any nonzero 
$x_\gamma\in E_\gamma$ and $x_\delta \in E_\delta$.

To complete part \textup{(iii)}, observe that exactness of (\ref{SESsemiram}) is a special case of 
Prop.~\ref{prop:TKvsSKforE}. 
For, as $d_E = 1$,
we have 
$\Nrd_E(E_0^*) = \wt N(E_0^*) = N_{E_0/T_0} ( \Nrd_{E_0}(E_0^*)) = N_{E_0/T_0}(E_0^*)$.

We now turn to  \textup{(iv)}. Suppose that $E$ has maximal graded subfields $L$ and $K$, respectively unramified and 
totally ramified over $T$. Then, invoking the fundamental equality $ {[E_0:T_0]\,|\Gamma_E:\Gamma_T|  = [E:T]} $, we have
\begin{align*}
[L:T]\,&[K:T] \, = \, [L_0:T_0]\,|\Gamma_K:\Gamma_T| \, \le [E_0:T_0]\,|\Gamma_E:\Gamma_T|  \\ 
\, &= \,[E:T] \,
= \, \opn{ind}(E)^2\, = \, [L:T]\,[K:T].
\end{align*}
Hence, equality must hold throughout, showing that $E_0 = L_0$, a field with $[E_0:T_0] = [L_0:T_0] = [L:T]= \opn{ind}(E)$
and $\Gamma_E = \Gamma_K$, so that $|\Gamma_E:\Gamma_T| = |\Gamma_K:\Gamma_T| = [K:T] = \opn{ind}(E)$.
Thus, $E$ is semiramified, so part (iii) applies here.

We claim that the map $G\wedge G\to\mathcal H_\mu(G,E_0^*)$ is trivial. 
Since $\Gamma_E = \Gamma_K$, the arbitrary nonzero elements $x_\gamma \in E_\gamma$ and $x_\delta \in E_\delta$
may each be chosen to be in $K^*$; 
As $K$ is commutative  the map $\alpha$ in 
diagram 1 is trivial;  hence the map into  $H_\mu(G,E_0^*)$ is also trivial, as claimed.
Therefore the exact sequence (\ref{SESforH_mu}) in \textup{(iii)} yields $\mathrm{TK}_1(E)\cong \mathcal H_\mu(G,E_0^*)$.
Correspondingly, the analogous argument for $\SK(E)$ given in \cite[Cor.~3.6(iv)]{hazrat-wadsworth-2011} gives 
$\SK(E) \cong  \widehat H^{-1}(G,E_0^*)$.  Therefore, exact sequence (\ref{SESH^-1forTK_1}) follows from 
exact sequence (\ref{SESsemiram}), completing the proof of  \textup{(iv)}.

It remains to prove \textup{(v)}.  Suppose that $\Gamma_T \cong\zz$,  Then $d_E = 1$ by 
\cite[Prop.~8.49]{tignol-wadsworth-2015}.  Moreover, $\Gamma_E/\Gamma_T \wedge \Gamma_E/\Gamma_T$
is trivial, as $\Gamma_E/\Gamma_T$ is a cyclic group.  Hence, the exact lower row of Diagram~1 yields
$\TK(E) \cong E^{(\mu)}/\mathcal I_E$. With this, exact sequence (\ref{ES:TK_1(E)Gamma_cyclic}) follows 
from the exact top row of Diagram 1,  and short exact sequence (\ref{SES:TK_1(E)Gamma_cyclic}) follows from the 
short exact column of the diagram.   
\end{proof}

\section{The obstruction group for the congruence theorem for $\operatorname{TK}_1(D)$}\label{sec:obstruction-group}

We next turn to the valued setting.  Let $D$ be a division ring finite-dimensional over its center $K$,  such that 
$D$ has a valuation $v$ whose restriction to $K$ is Henselian, and $D$ is tame over $K$ with respect to $v$.
We use the notation and terminology introduced in Section 2. 

The Congruence Theorem for $\SK(D)$ says that 
\begin{equation}
\label{SKcongth}
(1+M_D) \cap D^{(1)} \, \subseteq D'.
\end{equation}
This was proved by Platonov in \cite{platonov-1976} for $v|_K$ complete discrete, and in full generality for $v|_K$
Henselian (and $D$ tame) in \cite{hazrat-wadsworth-2011}.

In \cite{khanh-khoa-2025}, Khanh and Khoa identified the obstruction to a congruence theorem for 
$\operatorname{TK}_1(D)$
as the group 
$$
 \bH\,=((1+M_D)\cap D^{(\mu)})\big/((1+M_D)\cap D').
$$
The congruence theorem for $\operatorname{TK}_1(D)$ holds precisely when $\mathbf{H}=1$. 
It was shown in \cite{khanh-khoa-2025} that $\bH$ is trivial if $\charr(\ov K) = 0$ or if 
$\charr(\ov K) = p>0$ and $K$ contains no primitive $p$-th root of unity.
We now fully determine $\bH$.

In light of the Congruence Theorem for $\SK(D)$, recalled in (\ref{SKcongth}),
we can restate $\bH$ as 
\begin{equation}
\label{Hrestated}
        \mathbf{H}\,=\,\big((1+M_D)\cap \Dmu\big)\big/\big((1+M_D)\cap D^{(1)}\big).
\end{equation}

\begin{proposition}
\label{prop:obstruction-norm}  Let $N(\bH) = \Nrd_D\big((1+M_D)\cap \Dmu\big)$. Then,
$$
\bH\, \cong N(\bH) \, = \, (1+M_K)\cap\mu_K.
$$
\end{proposition}
\begin{proof}
The kernel of the reduced norm map $\Nrd_D\colon (1+M_D)\cap \Dmu\to K^*$ is 
$(1+M_D)\cap \Dmu\cap D^\upone
= (1+M_D)\cap D^\upone$.  Therefore, in view of (\ref{Hrestated}),  $\Nrd_D$ induces an isomorphism 
$\bH\, \cong N(\bH)$. 
 
To prove the equality in the proposition, recall from (\ref{eq:Nrd(1+M_D)}) above  that 
$\operatorname{Nrd}_D(1+M_D)=1+M_K$. Now, take any 
$d\in(1+M_D)\cap \Dmu$. Since $d\in1+M_D$, we have $\operatorname{Nrd}_D(d)\in1+M_K$. Since $d\in \Dmu$, 
some power of $d$ lies in $D'$, say $d^m\in D'\subseteq D^{(1)}$. Then $\operatorname{Nrd}_D(d)^m=1$, so $\operatorname{Nrd}_D(d)\in\mu_K$. Thus $N(\mathbf{H})\subseteq(1+M_K)\cap\mu_K$.

Conversely, take $c\in(1+M_K)\cap\mu_K$, say $c^\ell=1$. Choose $b\in1+M_D$ with $\operatorname{Nrd}_D(b)=c$. 
Then $\operatorname{Nrd}_D(b^\ell)=1$, so $b^\ell\in D^{(1)}$. Since $D^{(1)}/D'$ is torsion, $b^{\ell k}\in D'$ for 
some $k\geq1$, and hence $b\in \Dmu$. Therefore $c=\operatorname{Nrd}_D(b)\in N(\mathbf{H})$, proving the 
reverse inclusion.
\end{proof}

We now characterize $(1+M_K)\cap\mu_K$ as $\mu_K[p]$ or trivial.  Recall that for  $p$  any prime number,
$\mu_K[p]$ denotes the $p$-primary component of $\mu_K$.  Moreover, if  $\mu_K[p]$ is nontrivial,
then necessarily ${\charr(K) \ne p}$. If this holds, then  $\mu_K[p]$ is either a finite cyclic $p$-group or an 
infinite group which is a nested union of finite cyclic $p$-groups; all these possibilities occur for 
suitable choices of $K$.

\begin{lemma}
\label{lem:roots-of-unity-s}
\textit{The following assertions hold:}
\begin{enumerate}[font=\normalfont]
\item[(i)] \textit{If $\operatorname{char}(\overline K)=p>0$, then $(1+M_K)\cap\mu_K=\mu_K[p]$.}
\item[(ii)] \textit{If $\operatorname{char}(\overline K)=0$, then $(1+M_K)\cap\mu_K=\{1\}$.}
\end{enumerate}
\end{lemma}
\begin{proof}
Suppose first that $\operatorname{char}(\overline K)=p>0$. Take any $c\in\mu_K[p]$. Then the image $\overline c$ of 
$c$ in $\overline K$ lies in $\mu_{\overline K}[p]=\{1\}$; hence $c\in1+M_K$. Thus 
$\mu_K[p]\subseteq(1+M_K)\cap\mu_K$. 
Suppose this inclusion were strict. Then there is an $a\in(1+M_K)\cap\mu_K$ whose order $r$ in $\mu_K$ is not a 
power of $p$. Hence there is a prime $q\ne p$ dividing $r$. Let $b=a^{r/q}$, so $b\in1+M_K$, $b^q=1$, and $b\ne1$. 
Thus $b$ is a root of $(X^q-1)/(X-1)\in V_K[X]$. Hence $\overline b=1$ in $\overline K$, but also $\overline b$ is a root 
of $(X^q-1)/(X-1)$ in $\overline K[X]$. Since $q\ne \operatorname{char}(\overline K)$, the root $1$ is a simple root 
of $X^q-1$ and hence cannot be a root of $(X^q-1)/(X-1)$. This contradiction proves (1).

If $\operatorname{char}(\overline K)=0$, the same argument applies: If $(1+M_K)\cap\mu_K$ is nontrivial, it contains 
an element $b$ of prime order $q$. Then $\overline b=1$, while $\overline b$ is a primitive $q$-th root of unity in 
$\overline K$, contradiction. This proves (2).
\end{proof}

\begin{corollary}
\label{cor:obstruction}{\phantom{.}}
\begin{enumerate}
 [font=\normalfont]
\item[(i)] If $\operatorname{char}(K)=\operatorname{char}(\overline K)$, then $\mathbf{H}=1$.
\item[(ii)] If $\operatorname{char}(K)=0$ and $\operatorname{char}(\overline K)=p>0$, then $\mathbf{H}=\mu_K[p]$.
\item[(iii)] In all characteristics for $K$ and $\ov K$,   we have $\mu_K \cong \bH \times \mu_{\ov K}$.  
\end{enumerate}
\end{corollary}
\begin{proof}
Parts (i) and (ii)  follow immediately from  Prop.~\ref{prop:obstruction-norm} and Lemma~\ref{lem:roots-of-unity-s}.
For (iii) note that if $\charr(K) = \charr(\ov K)$, then $\bf H = 1$ while,  as the valuation on $K$ is Henselian, 
$\mu_K \cong\mu_{\ov K}$; so (iii) holds in this case.  If $\charr(K) = 0$ while $\charr(\ov K) = p >0$, then 
$\mu_{\ov K}[p] = 1$, while, as $v$ on $K$ is Henselian, the prime-to-$p$ parts of $\mu_K$ and $\mu_{\ov K}$
are isomorphic.  Hence, 
$
\mu_K \cong \mu_K [p] \times \mu_{\ov K} \cong \bH \times \mu_{\ov K}.
$
Thus, (iii)~holds in all cases. 
\end{proof}

Case (i) of this corollary was proved by Motiee in \cite[Th.~9]{motiee-2013} under the added 
hypothesis that $\charr(\ov K)\nmid\opn{ind}(D)$.

\begin{remark}
The preceding corollary sharpens the obstruction result of \cite{khanh-khoa-2025}. It shows that $\mathbf{H}$~depends 
only 
on the valued center $K$, and not on the particular tame division algebra $D$. In particular, if $K$ is the $2$-adic 
field $\mathbb Q_2$, then  $\mathbf{H}\cong \mu_{\mathbb Q_2}[2]=\{\pm1\}$ regardless of the parity of 
$[D:\mathbb Q_2]$.
This corrects an error in 
\cite[Prop.~3(2)]{khanh-khoa-2025}.  But note that if $[D:\mathbb Q_2]$ is even, then 
$\bH \ne \{\pm 1\}((1+M_D)\cap D')$, even though
$\bH \cong \{\pm 1\}$.  See Example~\ref{ex:p-adic}(ii)  below.
\end{remark}

\section{A valued exact sequence for $\operatorname{TK}_1(D)$}\label{sec:valued-exact-sequence}

The graded results of Section~\ref{sec:graded-tk1} can be applied to $E=\gr(D)$.  The remaining point in passing from 
a tame valued division algebra $D$ to $\gr(D)$ is the obstruction $\bH$ computed in Section~\ref{sec:obstruction-group}.  
The next result identifies this obstruction as the kernel of the natural map $\TK(D)\to\TK(\gr(D))$.  Throughout this 
section, 
$D$ is a division algebra finite-dimensional over its center $K$, with a valuation~$v$ on $D$ such that $v|_K$ is 
Henselian and $D$ is tame with respect to $v$.

\begin{theorem}
\label{thm:valued-tk1-exact-sequence}
There is a short exact sequence
\begin{equation}
\label{SESforH}
        1\,\longrightarrow \,\mathbf{H}\,\longrightarrow \,\operatorname{TK}_1(D)\,\longrightarrow\, 
        \operatorname{TK}_1(\gr(D))\,\longrightarrow \,1.
\end{equation}

\end{theorem}
\begin{proof}
We will use  the property  given in (\ref{eq:Nrdcompatibility}) above:  $\operatorname{Nrd}_{\gr(D)}(\widetilde a)=\widetilde{\operatorname{Nrd}_D(a)}$ for all $a\in D^*$.

 Consider the diagram
$$
\begin{tikzcd}[column sep=large,row sep=large]
1 \arrow[r]
& (1+M_D)\cap D' \arrow[r] \arrow[d,hook]
& D' \arrow[r,"\rho"] \arrow[d,hook]
& (\gr(D))' \arrow[r] \arrow[d,hook]
& 1 \\
1 \arrow[r]
& (1+M_D)\cap \Dmu \arrow[r]
& \Dmu  \arrow[r,"\rho"]
& \gr(D)^{(\mu)} \arrow[r]
& 1 .
\end{tikzcd}
$$

The columns are all inclusion maps.
The top row is exact because the canonical map ${\rho\colon D^*\to(\gr(D))^*}$ is onto and $\ker\rho=1+M_D$. 

We next show that 
$\rho(\Dmu)=\gr(D)^{(\mu)}$, which yields exactness of the bottom row.
If ${c\in \Dmu}$, then $c^n\in D'$ for some $n\ge 1$, so $\widetilde c^{\,n}\in\rho(D')=(\gr(D))'$, 
and hence $\widetilde c\in \gr(D)^\upmu$. Thus 
${\rho(\Dmu)\subseteq \gr(D)^\upmu}$. Conversely, take $a\in D^*$ with 
$\widetilde a\in \gr(D)^\upmu$. Then $\widetilde a^{\,k}\in(\gr(D))'$ 
for some $k\geq1$, so $(\Nrd_{\gr(D)}(\widetilde a))^k = \Nrd_{\gr(D)}(\widetilde a^k)=1$. 
If $\operatorname{char}(\overline K)=p>0$, 
let $m$ be the prime-to-$p$ part of~$k$; since $\mu_{\gr(K)}$ has no $p$-torsion, $(\Nrd_{\gr(D)}(\widetilde a))^m=1$. 
Hence $\widetilde{\Nrd_D(a^m)}=1$, so $\operatorname{Nrd}_D(a^m)\in1+M_K$. Since $p\nmid m$ and 
$K$ is Henselian, the $m$-th power map on $1+M_K$ is onto. Choose $b\in1+M_K$ with $b^m=\Nrd_D(a^m)$, and 
choose $c\in1+M_D$ with $\Nrd_D(c)=b$. Then $\Nrd_D((ac^{-1})^m) =\Nrd_D(a^m)\Nrd_D(c)^{-m}=1$, 
so $(ac^{-1})^m\in D^{(1)}$. Since $D^{(1)}/D'$ is torsion, $ac^{-1}\in \Dmu$. As $\rho(c)=1$, we have 
$\widetilde a=\rho(ac^{-1})\in\rho(\Dmu)$.

If instead $\operatorname{char}(\overline K)=0$, the same argument works with $m=k$. Hence, 
$\rho(\Dmu)= \gr(D)^\upmu$
in all cases.
By taking cokernels of the vertical maps in the diagram and invoking the 5-Lemma we obtain the desired short exact 
sequence. \end{proof}

By applying this theorem and preceding results, we obtain a description of  
$\TK(D)/\SK(D)$:

\begin{proposition} 
\label{prop:TKvsSKforD}
$$
\TK(D)/\SK(D)\  \cong \ \bH\times 
 \big[\mu_{\ov K} \cap (N_{Z(\ov D)/\ov K}(\Nrd_{\ov D}(\ov D^*)))^{d_D} \big],
$$ 
where $d_D = \opn{ind}(D)\big/\big(\opn{ind}(\ov D)[Z(\ov D): \ov K]\big) \in \nn$.
\end{proposition} 

\begin{proof}
Let $E =\gr(D)$.  From the exact sequence (\ref{SESforH}) and the isomorphism $\SK(D) \cong \SK(E)$
of \cite[Th.~4.8]{hazrat-wadsworth-2011}, we obtain the following exact sequence:
$$
1\, \longrightarrow \, \bH \, \longrightarrow \, \TK(D)/\SK(D)\, \longrightarrow \, 
\TK(E)/\SK(E)\, \longrightarrow \, 1
$$
Since $E_0 \cong \ov D$, $T_0 \cong \ov K$, $\Nrd_{E_0} = \Nrd_{\ov D}$, $\ind (E) = \ind(D)$, and 
$\ind(E_0) = \ind(\ov D)$,
we have $d_D = d_E$ as given in (\ref{E_0Nrdformula}).  By applying these isomorphisms and equalities along with  Prop.~\ref{prop:TKvsSKforE},
as well as the formula in (\ref{E_0Nrdformula}), 
 the right term in this exact sequence becomes 
 $\mu_{\ov K} \cap\big(N_{Z(\ov D)/\ov K}(\Nrd_{\ov D}(\ov D^*))\big)^{d_D}$.
 When $\bH$~is trivial, this yields the proposition.  When $\bH$~is nontrivial, we have from Cor.~\ref{cor:obstruction}
 that $\charr(K) = 0$, $\charr(\ov K) = p > 0$ for some prime $p$, and $\bH \cong\mu_K[p]$.  Since the right term in the 
 exact sequence is a subgroup of $\mu_{\ov K}$, it is abelian torsion, but with no $p$-torsion.  Therefore, the exact
 sequence is then split, and the proposition holds in this case also. 
\end{proof}

From the preceding theorem and proposition, we  deduce formulas for  $\TK(D)$ from 
the corresponding ones for $\TK(\gr(D))$ given in 
Theorem~\ref{theorem:graded-tk1-cases} and the information about $\SK(D)$ given 
in~\cite{hazrat-wadsworth-2011}.  We continue to assume that $D$ is finite-dimensional and  tame over its 
Henselian center~$K$.

\begin{corollary}{\phantom{.}}
\label{cor:TK_1(D)cases}
\begin{enumerate}[font=\normalfont]
\item[(i)] 
If $D$ is unramified over $K$, then 
there is a short exact sequence:
$$
1 \,\longrightarrow\, \bH\,\longrightarrow \, \TK(D)\, \longrightarrow \,  \TK(\overline D)\, \longrightarrow \,1
$$  
\item[(ii)] 
If $D$ is totally ramified over $K$, then $\TK(D)\cong\mu_K/\mu_K(e)$, where $e=\exp(\Gamma_D/\Gamma_K)$.
\item[(iii)] 
Suppose $D$ is semiramified, i.e., $[\ov D:\ov K] = |\Gamma_D:\Gamma_K| = \opn{ind}(D)$ and $\ov D$ is a field.
Then $\ov D$~is abelian Galois over $\ov K$, and $\opn{Gal}(\ov D/\ov K) \cong \Gamma_D/\Gamma_K$.  Let 
$G = \opn{Gal}(\ov D/\ov K)$.  There is an exact sequence:
$$
G \wedge G \, \longrightarrow \, \widehat H^{-1}(G, \ov D^*) \, \longrightarrow \, \TK(D)\, \longrightarrow \, 
\bH \times \big(\mu_{\ov K} \cap N_{\ov D/\ov K}(\ov D^*)\big)\, \longrightarrow \, 1 
$$ 
\item[(iv)] 
Suppose that $D$ has a maximal subfield $R$ that is unramified over $K$ and another maximal subfield $S$
that is totally ramified over $K$.  Then $D$ is semiramified over $K$.  Let $G = \opn{Gal}(\ov D / \ov K)$.  Then 
there is a short exact sequence:
$$
1\,\longrightarrow \,  \widehat H^{-1}(G, \ov D^*) \, \longrightarrow \, \TK(D)\, \longrightarrow \, 
\bH \times \big(\mu_{\ov K} \cap N_{\ov D/\ov K}(\ov D^*)\big)\, \longrightarrow \, 1 
$$
\end{enumerate}
\end{corollary}

\begin{proof}
Let $E = \gr(D)$ and $T = \gr(K) = Z(E)$.  Note that for each case (i) - (iv) here, $E$ falls into the 
corresponding case of  Th.~\ref{theorem:graded-tk1-cases}.

For (i), suppose $D$ is unramified; so, $E$ is unramified. The short exact sequence in (i) follows from 
Th.~\ref{thm:valued-tk1-exact-sequence}, as $\TK(\ov D) \cong \TK(E_0) \cong \TK(E)$ by 
Th.~\ref{theorem:graded-tk1-cases}(i).

Next, for (ii) assume that $D$ is totally ramified.  Then, $E$ is totally ramified.  By 
Th.~\ref{theorem:graded-tk1-cases}(ii), $\TK(E) \cong \mu_{T_0}/ \mu_{T_0}(e)\cong \mu_{\ov K}/\mu_{\ov K}(e)$, 
where $e =\opn{exp}(\Gamma_E/\Gamma_T) = \opn{exp}(\Gamma_D/\Gamma_K)$.  Recall that $e$ 
is prime to $\charr(\ov K)$ by \cite[Prop.~7.72]{tignol-wadsworth-2015}.  Hence, $\mu_K(e) \cong \mu_{\ov K}(e)$
as the valuation on $K$ is Henselian.
  If ${\charr(K) = \charr(\ov K)}$,  then $\bH = 1$ and $\mu_K\cong \mu_{\ov K}$ as the valuation is Henselian;
  so, by Th.~\ref{thm:valued-tk1-exact-sequence}, ${\TK(D) \cong \TK(E) \cong \mu_K/\mu_K(e)}$.  
  On the other hand, if $\charr(K) = 0$ while $\charr(\ov K) = p >0$, 
  then $\mu_K \cong \mu_K[p] \times \mu_{\ov K}$.  
   Moreover, $\bH \cong \mu_K[p]$ has $p$-primary torsion
   while $\TK(E)$ has no $p$-torsion. It follows that the short exact sequence (\ref{SESforH}) is 
   split exact.
Therefore,
  $$
  \TK(D) \,\cong \,\bH \times \TK(E) \,\cong\, \mu_K[p] \times \mu_{\ov K}/\mu_{\ov K}(e)\, \cong \,\mu_K/\mu_K(e).
  $$
  Thus, (ii) holds in all cases.

Now, for (iii), assume that $D$ is semiramified.  Then, $E$ is also semiramified.  Hence, ${d_D = d_E = 1}$, 
and as $\ov D$ is a field, Prop.~\ref{prop:TKvsSKforD} can be restated for the present situation as the short exact 
sequence:
\begin{equation}
\label{semiramSES}
1 \,\longrightarrow\, \SK(D) \, \longrightarrow  \TK(D) \, \longrightarrow \,  
\bH \times \big(\mu_{\ov K} \cap N_{\ov D/\ov K}(\ov D^*)\big)\, \longrightarrow \, 1
\end{equation}
By \cite[(4.4)]{hazrat-wadsworth-2011}, we have the exact sequence: 
\begin{equation}
\label{semiramSK_1ES}
G \wedge G \, \longrightarrow \, \widehat H^{-1}(G, \ov D^*) \, \longrightarrow \, \SK(D)\, \longrightarrow \, 1
\end{equation}
By splicing together exact sequences (\ref{semiramSES}) and (\ref{semiramSK_1ES}), we obtain the 
exact sequence of (iii).

Finally, suppose $D$ is as in (iv). Then, $D$ is semiramified as shown in the proof of 
 \cite[Cor.~4.10(iv)]{hazrat-wadsworth-2011} (or by using the valued version of the 
 argument above for Th.~\ref{theorem:graded-tk1-cases}(iv)); so part (iii) applies here.
 In the context of (iv), it was shown in \cite[Cor.~4.10(iv)]{hazrat-wadsworth-2011}
 that  $\SK(D) \cong \widehat H^{-1}(G, \ov D^*)$.
 From this isomorphism and (\ref{semiramSES}) we obtain the short exact sequence  of (iv).
\end{proof}

Part (i) of the corollary was proved by Motiee in \cite[Th.~10(1)]{motiee-2013}
under the added assumptions that $\charr(\ov K) = \charr(K)$ (so $\bH = 1$) and 
$\charr(K) \nmid \opn{ind}(D)$.  He also proved part (ii) of the corollary in \cite[Th.~10(2)]{motiee-2013}
under the added assumption that $\charr(\ov K) = \charr(K)$.

\medskip

\begin{examples}{\phantom{.}}

\label{ex:p-adic}
{}\!\!\!(i)
Let $K$ be a local field, i.e., $K$ is a finite-degree extension of some $p$-adic field 
$\qq_p$.  Then, $K$~has a complete discrete  (so Henselian) valuation  $v$ extending the
$p$-adic valuation on $\qq_p$, and $\ov K$ is the finite field  $\ff_{p^k}$ of cardinality $p^k$
for some $k\ge 1$.  Assume that $v$ is normalized so that $\Gamma_K = \zz$.  For some $n \ge 1$,
let $R$ be the unique (up to isomorphism) unramified field extension of $K$ of degree $n$.  So, 
$\ov R = \ff_{p^{kn}}$ and $R$ is cyclic Galois over $K$, as $\ov R$ is cyclic Galois over $\ov K$.
Choose a generator $\sigma$ of~$\opn{Gal}(R/K)$;  then $\sigma$ induces a $\ov K$-automorphism of 
$\ov R$, and $\opn{Gal}(\ov R/\ov K) = \langle \ov \sigma\rangle$. Choose any $\pi \in M_K$ with 
$v(\pi) = 1$.  Let $D = (R/K, \sigma, \pi)$, which is the cyclic algebra over $K$ generated by $R$ and 
an element $j$, such that $jrj^{-1} = \sigma(r)$ for all $r\in R$ and $j^n = \pi$.  So, $R$ is a maximal
subfield of~$D$.  Let $S = K[j] \subseteq D$; so, $S$ is a field (as $X^n - \pi$ is irreducible in $K[X]$)
which is a maximal subfield of~$D$, and is totally ramified over~$K$.  Furthermore, $D$ is a division 
ring, as $\pi^\ell$ is not in the image of the norm map $N_{R/K}$ for $1 \le \ell < n$.  Thus, we are in 
the situation of Cor.~\ref{cor:TK_1(D)cases}(iv).  So, $\ov D = \ov R$, and as $G = \opn{Gal}(\ov D/\ov K)$
is a cyclic group, Hilbert 90 yields $\widehat H^{-1}(G, \ov D^*) = 1$.  Moreover,  as $\ov D$ is finite, 
$N_{\ov D/\ov K}$ maps $\ov D^*$ onto $\ov K^*$.  Thus, Cor.~\ref{cor:TK_1(D)cases}(iv) yields
$$
\TK(D) \, \cong\, \bH \times \big(\mu_{\ov K} \cap  N_{\ov D/\ov K}(\ov D^*)\big) \,
\cong \mu_K[p] \times \mu_ {\ov K} \, \cong \,\mu_K.
$$
In fact, every division algebra index $n$ with center a local field $K$ is of the type described here,
and is determined up to isomorphism by the choice of $\sigma$.  

{}\!\!\!(ii) 
As a specific case of (i), let $K = \qq_2$, and let $D$ be the quaternion algebra 
$\big (-3, 2\big/ \qq_2 \big)$, i.e., the algebra generated over $\qq_2$ by $i$ and $j$, 
such that $i^2 = -3$, $j^2 = 2$, and $ij = -ji$.   As $\mu(\ov{\qq_2}) = 1$, we have 
$\TK(D) \cong \bH \cong \{1, -1\}$.  But note that $-1 = [i,j] \in D'$ and $-1 = 1 + (-2) \in 1+M_D$.  Hence, while 
$\bH \cong \{1,-1\}$, both $1$ and $-1$ represent the trivial element  of $\bH$.
 To identify the nontrivial element 
of $\bH$, consider $1+j\in 1+M_D$.  Note that 
$$
\Nrd_D(1+j) \,=\, N_{\qq_2[j]/\qq_2}(1+j) \,=\, (1 + j)(1-j) \,=\, -1.
$$
Thus, $1+j \in \Dmu \setminus D'$, so $(1+j)((1+M_D)\cap D')$ is the nontrivial element of $\bH$.
\end{examples}

\section{Stability of $\operatorname{SK}_1$ and $\operatorname{TK}_1$}\label{sec:stability}

The exact diagram of Section~\ref{sec:graded-tk1} describes $\TK(E)$ for a graded division algebra $E$
in terms of $E_0$ and $\Gamma_E$
and the conjugation action of $E^*$ on $E_0$. In this section we approach $\TK(E)$ from a different 
direction:  We prove for it the \lq\lq stability theorem" saying that $\TK(E) \cong \TK(Q)$, where 
$Q = q(E)$ is the division ring of central quotients of $E$.  At the same time, we give a new proof of 
the corresponding stability result for $\SK(E)$, which is substantially simpler than the original proof in 
 \cite[Th.~5.7]{hazrat-wadsworth-2011}. Applications of  
stability for $\TK(E)$ will be given in Section~\ref{sec:stability-consequences}.

The original stability theorem (and the origin of the term \lq\lq stability") was Platonov's result 
in~\cite{platonov-stability-1976} that if $D$ is a division algebra finite-dimensional over its center $K$, 
then ${\SK(D(x)) \cong \SK(D)}$, where $D(x) = q(D[x]) = D[x] \otimes_{K[x]}K(x)$.  Later Platonov and 
Yanchevski\v{\i} generalized this result in~\cite{platonov-yanchevskii-1979} by giving calculations for 
$\SK(D(x;\varphi))$ in terms of data for 
$D$ and the automorphism~$\varphi$ of~$D$, where $D(x; \varphi)$ is the division ring of 
central quotients of the twisted polynomial ring $D[x;\varphi]$.  It was the recognition 
that the work in \cite{platonov-yanchevskii-1979} could be reformulated and clarified as making calculations
for $\SK$ of the graded division algebra $E = D[x, x^{-1};\varphi]$ and then showing that 
$\SK(E) \cong \SK(q(E))$ that led to the latter step being called a stability theorem.

As in the preceding section, let $E$ be a graded division algebra finite-dimensional over its 
center~$T$; let $q(T)$ be the quotient field  of the integral domain $T$, and let 
$Q= E \otimes_T q(T)$, the  quotient division ring of $E$.

\begin{lemma}
\label{lem:Q'capE*}
$$
Q' \cap E^* \, = \, E'.
$$
\end{lemma}

\begin{proof}
By \cite[p.~124]{hazrat-wadsworth-2011} or \cite[Sec.~3.2.5]{tignol-wadsworth-2015} there is a group 
homomorphism $\lambda\colon Q^* \to E^*$ such that $\lambda |_{E^*} = \opn{id}_{E^*}$. 
Then, $\lambda(Q') \subseteq E'$. 
So, if  $e \in Q' \cap E^*$, then $e = \lambda(e) \in E'$.  Thus, $Q' \cap E^* \subseteq E'$, and the reverse
inclusion is clear.
\end{proof}

It follows immediately from the lemma that 
\begin{equation}
\label{isoforTK_E}
E^*/E' \, \cong\, E^*Q'/Q'.
\end{equation} 
This isomorphism shows that in order to prove stability for $\TK$ it suffices to show that the torsion
part of $Q^* /Q'$ lies in $E^*Q'/Q'$.  The next proposition yields this.

\begin{proposition}
\label{prop:cokernel-torsion-free}
The factor group $Q^*/(E^*Q')$ is torsion-free.
\end{proposition}

Before proving the proposition, we recall an easy well-known lemma about permutation modules, for which we do not 
have a good reference.

\begin{lemma}
\label{lemma:permmod}
Let $\scB$ be a nonempty set, and let $P$ be the free abelian group with base $\scB$.  Let $G$ be any group which 
acts on $\scB$, and extend this action as usual to a $\zz$-linear action of $G$ on $P$.  Thus, $P$ is a permutation 
$G$-module.  Let 
$I_G(P) = \langle b - \sigma(b)\mid b\in \scB, \sigma \in G\rangle \subseteq P$.  Let $\scS$ be the set of orbits of $B$
 under the $G$-action, and for each orbit $\scO_s \in \scS$ choose any $d_s \in \scO_s$.  Then, 
$$
P \, = \,I_G(P) \,\textstyle{\bigoplus} \, \big(\underset{\scO_s \in \scS}\bigoplus \zz d_s\big).  
$$
\end{lemma} 

\begin{proof}
Let $\mathbb Z[\mathcal S]$ be the free abelian group with base $\mathcal S$. The orbit map $\mathcal B\to\mathcal S$ 
extends to an epimorphism $\pi\colon P\to\mathbb Z[\mathcal S]$, where $\pi(b)=\mathcal O_s$ whenever 
$b\in\mathcal O_s$. Since $b$ and $\sigma(b)$ belong to the same orbit, $I_G(P)\subseteq\ker\pi$. For the reverse 
inclusion, let $a=\sum_{b\in\mathcal B}n_b b\in\ker\pi$. For each orbit $\mathcal O_s\in\mathcal S$, we then have 
$\sum_{b\in\mathcal O_s}n_b=0$, and hence
$$
        \textstyle{\underset{b\in\mathcal O_s}\sum}n_b b
       \ =\,
        \textstyle{\underset{b\in\mathcal O_s\setminus\{d_s\}}\sum}
        n_b(b-d_s).
$$
Since $b$ and $d_s$ lie in the same orbit, there is $\sigma\in G$ such that $\sigma(d_s)=b$; therefore 
$b-d_s\in I_G(P)$. Thus $a\in I_G(P)$, proving that $\ker\pi=I_G(P)$. The homomorphism 
$\mathbb Z[\mathcal S]\to P$ given by $\mathcal O_s\mapsto d_s$ is a section of $\pi$; so the lemma follows.
\end{proof}

\begin{proof}[Proof of Proposition~\ref{prop:cokernel-torsion-free}]
We first recall the rank-one divisor construction of \cite[Rem.~5.1 and Prop.~5.3]{hazrat-wadsworth-2011}. Let $D$ be 
a division ring finite-dimensional over its center, let $\sigma\in\opn{Aut}(D)$ be such that $\sigma|_{Z(D)}$ has finite 
order, and let 
$\mathcal T$ be the twisted polynomial ring $D[x;\sigma]$ and $Q_\sigma = q(\mathcal T)=D(x;\sigma)$. 
Let $\mathcal B$ be the set of isomorphism classes of simple left $\mathcal T$-modules.
The group $\opn{Div}(\mathcal T)$ is the free abelian group with basis $\mathcal B$. For any $f \in \mathcal T$ with 
$f \ne 0$, the left $\scT$-module 
$\scT/\scT f$ has  finite dimension as  $D$-vector space; hence, it has finite length as a $\scT$-module.
The divisor $\delta(f)$ in $\opn{Div}(\scT)$ is defined by taking a composition series for $\scT/\scT f$ 
and letting $\delta(f)$ be the sum of the classes of each  factor module in the series.  By the Jordan-H\" older Theorem, 
$\delta(f)$ is independent of the choice of  composition series.  It is easy to check that $\delta(fg) = \delta(f) + \delta(g)$ 
for any $f,g\in \scT \setminus\{0\}$.  
One then extends $\delta$ to a well-defined map  $\delta\colon Q_\sigma^*\to\opn{Div}(\scT)$ given by 
$\delta(fs^{-1})=\delta(f)-\delta(s)$ for any  $f\in\mathcal T\setminus\{0\}$ and $s\in Z(\mathcal T)\setminus\{0\}$.
Clearly, $\delta$ is a group homomorphism.  Also, $\delta$ is surjective, since, as $\scT$ is a principal left ideal domain
(see, e.g., \cite[p.~3]{jacobson-1996})
each generator of $\opn{Div}(\scT)$ in $\mathcal B$ lies in $\opn{im}(\delta)$.  
Moreover, it is shown in \cite[Prop.~5.3]{hazrat-wadsworth-2011} that 
$$
        \ker\delta=D^*Q_\sigma',
        \qquad\text{so}\qquad
        Q_\sigma^*/(D^*Q_\sigma')\cong\opn{Div}(\mathcal T).
$$
Note also that as $\scT x$ is a maximal left ideal of $\scT$, we have 
$\delta(x) = [\scT/\scT x]\in \mathcal B$.
We will use this repeatedly with $D=Q_{r-1}$, $\mathcal T=\mathcal T_r=Q_{r-1}[x_r;\varphi_r]$, and $Q_\sigma=Q_r$, as 
defined in the next paragraphs.

We first assume  that $\Gamma_E$ is finitely-generated. If $\Gamma_E=0$, then $E=E_0$ and $Q=E$, so there is 
nothing to prove. Thus, assume $\Gamma_E$ has finite rank $n\ge 1$. Since $\Gamma_E$ is torsion-free, it has a basis 
$\varepsilon_1,\ldots,\varepsilon_n$ of $\Gamma_E$. For $0\leq r\leq n$, set
\begin{equation*}
        \Delta_r=\zz\varepsilon_1+\ldots+\zz\varepsilon_r,\qquad
        E_r=\underset{\gamma\in\Delta_r}{\textstyle{\bigoplus}}E_\gamma,
        \qquad
        Q_r=q(E_r).
\end{equation*}

Thus $E_0$ is the degree-zero division ring, $E_n=E$, and $Q_n=Q$. 
Note that $E_r$ is a sub-graded division ring of $E$ with $(E_r)_0 = E_0$ and $\Gamma_{E_r} = \Delta_r$.  
Moreover, $E_r$ is actually a graded division algebra.  To see this, 
let $T_r = E_r \cap T$, which is a graded field with $T_r \subseteq Z(E_r)$ and 
$\Gamma_{T_r}  = \Gamma_{E_r} \cap \Gamma _T$. So,
$\big | \Gamma_{E_r}: \Gamma_{T_r}\big | = \big |(\Gamma_{E_r} + \Gamma_{T}): \Gamma_{T}\big| 
\le |\Gamma_E :\Gamma_T| $.
 Then,
$$
[E_r: Z(E_r)] \,\le \, [E_r: T_r] \,=\, [E_0:T_0] \,\big | \Gamma_{E_r}: \Gamma_{T_r}\big | \, 
\le \, [E_0:T_0]\, |\Gamma_E :\Gamma_T| \, =\, [E:T] \,<\, \infty,
$$
as needed for $E_r$ to be a graded division algebra.
Note also that the inclusion $E_r \subseteq Q$ extends canonically  to an injection 
$Q_r \hookrightarrow Q$.  Thus we may view $Q_r \subseteq Q$.

For $j$ with $1\leq j\leq n$, choose and fix $x_j\in E_{\varepsilon_j}\setminus\{0\}$. 
When $j \ge r$, let $\varphi_j$ be the inner automorphism $\opn{int}(x_j)$ of $Q$.  
As each  $x_j\in E^*$, $\varphi_j$ restricts to a (graded) automorphism of $E_r$
and of $E_{r-1}$, hence also to an automorphism of $Q_{r-1}$; we denote these
restrictions of $\varphi_j$ also by $\varphi_j$.   For each $j$ there is an integer $k_j >0$
such that $k_j\varepsilon _j \in \Gamma_T$.  Then $x_j^{k_j}= e_0z$ with 
$e_0 \in E_0$ and $z\in T$. Hence, $\varphi_j^{k_j} = \opn{int}(e_0)$, which is trivial
on $Z(E_0)$, so also on  $Z(E_r)$, $Z(E_{r-1})$ and on $Q_{r-1}$.
Note that $E_r=E_{r-1}[x_r,x_r^{-1};\varphi_r]$ and $Q_r=Q_{r-1}(x_r;\varphi_r)$.

Let $Q'=Q_n'$. For $0\leq r\leq n$, put 
\begin{equation}
\label{def:N_rM_rA_r}
N_r\,=\,Q_r^*\cap Q', \qquad  M_r\,=\,E_r^*\cap Q', \qquad \text{and} \qquad
A_r\,=\,Q_r^*/(E_r^*N_r).
\end{equation}
We prove by induction on $r$ that $A_r$ is torsion-free. Since $A_0=1$ and $A_n=Q^*/(E^*Q')$, this will prove the 
finitely-generated case of the proposition. 
To obtain this, we will prove the claim stated below.  Fix an $r \in \{1,\dots,n\}$, put $\mathcal T_r=Q_{r-1}[x_r;\varphi_r]$ and  
$\divr=\opn{Div}(\mathcal T_r)$, and let $\delta_r\colon Q_r^*\to\divr$ be the divisor 
map.  Let $\scB_r$ be the basis of $\divr$ consisting of isomorphism classes of simple left $\scT_r$-modules.
Put
$$
        G_r \, =\, \langle\varphi_{r+1},\ldots,\varphi_n\rangle
        \subseteq \opn{Aut}(Q_r).
$$
Here, if $r=n$, then the group $G_r$ is understood to be the trivial group. 
Each $\sigma \in G_r$ is conjugation by an element of $E^*$; hence $\sigma(x_r^iE_0^*) = x_r^iE_0^*$
for each $i \ge 0$, which implies that $\sigma$ restricts to an automorphism of $\scT_r$.  Hence, $G_r$
acts on $\scT_r$, and thereby on $\scB_r$.  This action extends to an action on $\divr$, making
$\divr$ into a permutation $G_r$-module.   We {\it {claim}} 
that $\delta_r(N_r)=I_{G_r}(\divr)$, where 
$I_{G_r}(\divr)= \langle \beta-\sigma(\beta)\mid \beta\in \scB_r,\ \sigma\in G_r\rangle
= \langle \eta-\sigma(\eta)\mid \eta\in \divr,\ \sigma\in G_r\rangle$. 

We prove the claim: 
For the inclusion $I_{G_r}(\divr)\subseteq\delta_r(N_r)$, take a 
generator $\eta-\sigma(\eta)$ of $I_{G_r}$, where $\eta\in \scB_r$. 
Choose $b\in Q_r^*$ with $\delta_r(b)=\eta$, and choose a monomial $u$ in 
$x_{r+1}^{\pm1},\ldots,x_n^{\pm1}$  with $\opn{int}(u)$ inducing $\sigma$ on $Q_r$. Then 
$bub^{-1}u^{-1}\in Q_r^*\cap Q'=N_r$, and $\delta_r(bub^{-1}u^{-1})=\eta-\sigma(\eta)$. For the reverse inclusion, 
If $r=n$, then $N_n=Q'$ and $\delta_n(N_n)=0=I_{G_n}(\divr)$, so there is nothing to prove. 
Assume $r<n$ and put
$$
        C\,= \,Q_r[x_{r+1},x_{r+1}^{-1},\ldots,x_n,x_n^{-1}; \varphi_{r+1}, \ldots, \varphi_n]\,\subseteq \, Q;
$$
so, $C$ is a graded division algebra and, as $E \subseteq C \subseteq Q$, we have  $q(C) = Q$.
Now, $N_r = Q_r^*\cap Q' \subseteq C^* \cap Q' $. Hence,  by Lemma~\ref{lem:Q'capE*} applied to $C$, 
we have  $N_r \subseteq C'$.
So, it suffices to show that 
$\delta_r$~maps any 
commutator in $C'$  to $I_{G_r}(\divr)$. 
For this, note that 
 every nonzero homogeneous element of~$C$ has the form $bu$, where $b\in Q_r^*$ and $u$ is a monomial in 
 $x_{r+1}^{\pm1},\ldots,x_n^{\pm1}$. Let $b,c\in Q_r^*$, and let $u,v$ be such monomials. Write 
 $\sigma_u=\opn{int}(u)|_{Q_r}$ and $\sigma_v=\opn{int}(v)|_{Q_r}$, and put $m=uvu^{-1}v^{-1}$. Then 
 $m\in E_0^*\subseteq Q_{r-1}^*$, whence $\delta_r(m)=0$. Since
$$
        [bu,cv]=b\,\sigma_u(c)\,m\,\sigma_v(b^{-1})\,c^{-1},
$$
it follows that
$$
        \delta_r([bu,cv])
        =
        \big(\delta_r(b)-\sigma_v(\delta_r(b))\big)
        +
        \big(\sigma_u(\delta_r(c))-\delta_r(c)\big),
$$
which lies in $I_{G_r}(\divr)$. Hence $\delta_r(N_r)\subseteq I_{G_r}(\divr)$, and therefore $\delta_r(N_r)=I_{G_r}(\divr)$, 
completing the proof of the claim.

We now continue the induction. Consider the following commutative diagram:
$$
\begin{tikzcd}[column sep=2.8em,row sep=2.4em]
1 \arrow[r]
& E_{r-1}^*/M_{r-1} \arrow[r] \arrow[d]
& E_r^*/M_r \arrow[r] \arrow[d]
& E_r^*/(E_{r-1}^*M_r) \arrow[r] \arrow[d,"\theta"]
& 1
\\
1 \arrow[r]
& Q_{r-1}^*/N_{r-1} \arrow[r]
& Q_r^*/N_r \arrow[r]
& Q_r^*/(Q_{r-1}^*N_r) \arrow[r]
& 1 .
\end{tikzcd}
$$
The rows are exact, since $M_{r-1}=E_{r-1}^*\cap M_r$ and $N_{r-1}=Q_{r-1}^*\cap N_r$. We now analyze  the 
right column:  Since $Q_r' \subseteq  Q_r^* \cap Q' = N_r$, we have 
$\opn{ker}(\delta_r) = Q_{r-1}^*Q_r' \subseteq
Q_{r-1}^*N_r$.  Hence, $\delta_r$ induces an isomorphism 
 $$
 \overline{\delta_r}\colon Q_r^*/(Q_{r-1}^*N_r) \,\overset{\sim}{\longrightarrow}\,
   \delta_r(Q_r^*)/\delta_r(Q_{r-1}^*N_r) \,=\, \divr/I_{G_r}(\divr).
   $$
   Since $E_r^* = E_{r-1}^*\langle x_r\rangle$, the domain of the map $\theta$ is the cyclic group generated by 
   the image of $x_r$.  Therefore, the image of $\ov {\delta_r} \circ \theta$ is 
   $\big(\zz\delta_r (x_r)+ I_{G_r}(\divr)\big)\big/I_{G_r}(\divr)$,
   which is an infinite cyclic group, by Lemma~\ref{lemma:permmod} applied to the permutation $G_r$-module
   $\divr$.  Since $\ov {\delta_r} \circ \theta$ maps its domain, a cyclic group,  onto an infinite cyclic group,
   the map $\theta$ must be injective. Therefore,  
the $5$-Lemma yields a short exact sequence of cokernels of the vertical maps in the diagram,
$$
        1\longrightarrow A_{r-1}\longrightarrow A_r\longrightarrow B_r
        \longrightarrow 1,
$$
where $A_r$ and $A_{r-1}$ are as in (\ref{def:N_rM_rA_r}) and 
$$
        B_r\,=\,Q_r^*\big/(Q_{r-1}^*E_r^*N_r).
$$
Since we saw above that $\opn{ker}(\delta_r) \subseteq
Q_{r-1}^*N_r$, we have $\opn{ker}(\delta_r) \subseteq Q_{r-1}^*E_r^*N_r = Q_{r-1}^*\langle x_r \rangle N_r$.
Hence,
$$
B_r \,\cong\,\delta_r(Q_r^*) \big/\delta_r\big(Q_{r-1}^*\langle x_r \rangle N_r\big)
\, =\, \divr\big/\big(\zz\delta_r (x_r)+ I_{G_r}(\divr)\big),
$$
which is isomorphic to a direct summand of $\divr$ by Lemma~\ref{lemma:permmod}; so, $B_r$ is torsion-free.

Now $A_0=1$ is torsion-free. If $A_{r-1}$ is torsion-free, then the exact sequence above and the torsion-freeness of 
$B_r$ imply that $A_r$ is torsion-free. Hence, by induction, $A_n=Q^*/(E^*Q')$ is torsion-free. This proves the
 finitely-generated case of the proposition.

Now assume that $\Gamma_E$ is not finitely-generated. Since $E$ is finite-dimensional over its center $T$, the group 
$\Gamma_E/\Gamma_T$ is finite.  Choose a finite set $\mathcal R\subseteq \Gamma_E$  such that 
$\mathcal R + \Gamma_T = \Gamma_E$.  Let~$\mathcal D$~be the directed  (by inclusion) set of all 
finitely-generated subgroups $\Delta\subseteq \Gamma_E$ such that $\mathcal R\subseteq \Delta$.  For 
$\Delta\in\mathcal D$, put
$$
        E_\Delta=\underset{\delta\in\Delta}{\textstyle{\bigoplus}} E_\delta,\quad\text{and}\quad
        T_\Delta=T\cap E_\Delta=\underset{\delta\in\Delta\cap\Gamma_T}{\textstyle{\bigoplus}} T_\delta.
$$
We claim that $T_\Delta = Z(E_\Delta)$.  Clearly $T_\Delta \subseteq  Z(E_\Delta)$.  But also, as 
$\Delta + \Gamma_T= \Gamma_E$ and each homogeneous component $E_\gamma$ of $E$ is a $1$-dimensional 
$E_0$ vector space, we have $E_\Delta^*T^* = E^*$.  Any element $z$ of $Z(E_\Delta)$ centralizes $E_\Delta^*$ and 
$T^*$, so all of $E^*$, hence all of $E$; so $z\in T$.  Thus $Z(E_\Delta) \subseteq T \cap E_\Delta = T_\Delta$, 
proving this claim.

It follows from this claim that $E_\Delta$ is finite-dimensional over its center $T_\Delta$.  Indeed, by the fundamental 
equality, $[E_\Delta:T_\Delta]=[E_0:T_0]\,|\Delta/(\Delta\cap\Gamma_T)| = [E:T]< \infty$.   Thus, the finitely-generated 
case applies to each $E_\Delta$.

Let $Q_\Delta = q(E_\Delta)$.   For every $\Delta, \Delta' \in \mathcal D$ with $\Delta \subseteq \Delta'$, it follows from 
the claim that 
$$
Z(E_\Delta) \subseteq Z(E_{\Delta'}) \subseteq Z(E), \quad \text{hence}
\quad  Q_\Delta \subseteq Q_{\Delta'} \subseteq Q.
$$
Therefore, in addition to the direct limit $E^* = \varinjlim_{\Delta\in\mathcal D} E_\Delta^*$, we have $Z(E)  = \varinjlim_{\Delta\in\mathcal D} Z(E_\Delta)$,  so 
${Q^* = \varinjlim_{\Delta\in\mathcal D} Q_\Delta^*}$ and hence, as every element of $Q'$ is a finite product of 
commutators, ${Q' = \varinjlim_{\Delta\in\mathcal D} {Q_\Delta}'}$.

To see that $Q^*/E^* Q'$ is torsion-free, take any $q \in Q^*$ with $q^m = a c$ for some $a\in E^*$ and $c\in Q'$ and 
positive integer $m$. Because of the direct limit properties just noted, it follows that there is an $\Omega \in \mathcal D$ 
with $q \in Q_\Omega^*$, $a\in E_\Omega^*$, and $c\in {Q_\Omega}'$.  As $q^m = ac \in E_\Omega^* {Q_\Omega}'$ 
and $Q_\Omega^*/E_\Omega^*{Q_\Omega}'$ is torsion-free as $\Gamma_{E_\Omega}=\Omega$ has finite rank, it 
follows that $q \in E_\Omega^*{Q_\Omega}' \subseteq E^*Q'$, showing that $Q^*/E^* Q'$ is indeed torsion-free.
This completes the proof of the proposition.
\end{proof}

\begin{theorem}
\label{thm:stability-sk1-tk1}
\textit{The natural inclusion $E^*\hookrightarrow Q^*$ induces isomorphisms}
$$
        \operatorname{SK}_1(E)\cong \operatorname{SK}_1(Q)
        \qquad\text{and}\qquad
        \operatorname{TK}_1(E)\cong \operatorname{TK}_1(Q).
$$

\end{theorem}
\begin{proof}
Consider first $\TK$.  Since $(Q^*/Q') \big/ (E^*Q'/Q') \cong Q^*/E^*Q'$, which is 
torsion-free by Prop.~\ref{prop:cokernel-torsion-free}, for the torsion subgroups we have 
$\tau(Q^*/Q') = \tau(E^*Q'/Q')$.
Hence, $Q^{(\mu)} \subseteq E^*Q'$.
But $E^*/E' \cong E^*Q'/Q'$ by (\ref{isoforTK_E}).  Thus,
\begin{equation}
\label{TK1stable}
\TK(E) \, = \, \tau(E^*/E') \, \cong \, \tau(E^*Q'/Q') \, = \,\tau(Q^*/Q') \, = \TK(Q).
\end{equation}
For $a\in \Emu$, this isomorphism in (\ref{TK1stable}) maps $aE' \mapsto aQ'\in Q^{(\mu)}/Q'$.
Recall from Sec.~2 that $\Nrd_E =  \Nrd_Q|_E$.  Hence, if $a\in \Eone$ then $a\in Q^{(1)}$.  Therefore,
the isomorphism maps $\SK(E)$ into $\SK(Q)$.  To see that this injection is also surjective, take any 
$b\in Q^{(1)}$.  Since $b \in Q^{(\mu)} \subseteq E^*Q'$ we may write $b = cd$  with $c \in E^*$ and 
$d\in Q'$.  Then, 
$$
1 \, = \,\Nrd_Q(b)\, =\,  \Nrd_Q(c) \cdot \Nrd_Q(d) \, = \Nrd_E(c) \cdot 1.
$$
Thus, $c\in \Eone$.  The isomorphism in (\ref{TK1stable}) maps $cE' \in \SK(E)$ to 
$cQ' = bdQ' = bQ' \in \SK(Q)$.  Thus, the isomorphism $\TK(E) \cong \TK(Q)$ restricts
to an isomorphism $\SK(E) \cong \SK(Q)$.
\end{proof}

\begin{remark}
\label{other-stability-proof}
The stability theorem for $\SK (E)$ was proved as Th.~5.7 in  \cite{hazrat-wadsworth-2011}, and the 
corresponding result for $\TK (E)$ proved here is actually deducible by a slight variation of the 
argument given there for $\SK$.
(Specifically, in the statement of Prop.~5.6 and formula (5.6) in \cite{hazrat-wadsworth-2011}, 
replace $Q^{(1)}$ with $Q^{(\mu)}$, and the proof goes through with very minor adjustments.)
However, the proof given here for stability of $\SK (E)$ is substantially shorter and simpler than 
the one in \cite{hazrat-wadsworth-2011}.
\end{remark}

\section{Some consequences of the stability of $\operatorname{TK}_1$}\label{sec:stability-consequences}

The stability theorem for $\TK$ proved in the preceding section opens that way to deducing $\TK$ results for 
graded division algebras from corresponding results for (not necessarily valued) division algebras.  We 
illustrate that approach in this section by proving graded analogues to the primary decomposition theorem 
and related results  proved by Motiee in \cite{motiee-2013}.

It is well-known (see, e.g.,  Draxl, \cite[Lemma~6, p.~160]{draxl-1983}) that if $D_1$ and $D_2$ are 
finite-dimensional 
division algebras with the same center~$K$ and 
$\opn{gcd}(\opn{ind}(D_1), \opn{ind}(D_2)) = 1$, then
\begin{equation}
\label{SK_1prod}
\SK(D_1 \otimes_K D_2) \, \cong \, \SK(D_1 ) \times \SK(D_2).
\end{equation}
The naive analogue to this for $\TK$ fails to hold, but the appropriate  analogue is given by 
Motiee's 
primary decomposition theorem for $\TK$.   We first recall that theorem.
Let $D$ be central simple division algebra over  $K$.  
 For any  prime number $p$, let $\opn{TK_{1,p}}(D)$ denote the $p$-primary component of the abelian torsion group 
 $\TK(D)$.  
Write $\ind(D) = p_1^{r_1}\ldots p_k^{r_k}$ where the $p_j$ are distinct primes, and let 
$D \cong D_{p_1} \otimes _K \ldots \otimes _K D_{p_k}$ be its primary decomposition, i.e., the $D_{p_j}$ are central 
division algebras over $K$ (unique up to isomorphism) with $\opn{ind}(D_{p_j}) = p_j^{r_j}$ for each $j$.  Motiee's 
primary decomposition theorem \cite[Th.~5]{motiee-2013} says that $\TK(D)$ has the primary decomposition 
\begin{equation}
\label{prdecompD}
        \mathrm{TK}_1(D)\,
        \cong\,
        \big(
        {\textstyle\bigoplus}_{j=1}^k \mathrm{TK}_{1,p_j}(D_{p_j})
        \big)
        \,{\textstyle\bigoplus}\,
        \big(
        \underset{p\notin\{p_1,\ldots,p_k\}}{\textstyle\bigoplus}
        \mu_{K}[p]
        \big).
\end{equation}
It follows at once from this that if $D_1$ and $D_2$ are central simple division algebras over $K$
with $\opn{gcd}(\opn{ind}(D_1), \opn{ind}(D_2)) = 1$, since the primary decomposition of 
$D_1\otimes_K D_2$ is the tensor product of the primary decompositions of $D_1$ and $D_2$,
\begin{equation}
\label{TKtensorprod}
       \TK(D_1\otimes_K D_2)\, \cong\,
        \Big(\underset{\,p\,\mid\,\opn{ind}(D_1)}{\textstyle\bigoplus}\opn{TK}_{1,p}(D_1)\Big)
        \,{\textstyle{\bigoplus}}\,
        \Big(\underset{\,p\,\mid\,\operatorname{ind}(D_2)}{\textstyle{\bigoplus}}\opn{TK}_{1,p}(D_2)\Big)
         \,{\textstyle{\bigoplus}}\,
        \Big(\underset{\,p\,\nmid\,\opn{ind}(D_1)\opn{ind}(D_2)}{\textstyle{\bigoplus}}       \mu_K[p]\Big).
\end{equation} 
Just as (\ref{SK_1prod}) reduces the task of computing $\SK(D)$ to the case where $\opn{ind}(D)$ is a
prime power, (\ref{prdecompD}) and (\ref{TKtensorprod}) give the analogous reduction for $\TK(D)$.

We now prove the version of Motiee's primary decomposition theorem for $\TK$ of a graded division algebra.
For this, let $E$ be a graded division algebra with center $T$, with $\opn{ind}(E) = p_1^{r_1}\ldots p_k^{r_k}$  
for distinct primes $p_1, p_2, \ldots, p_k$.  Then, by \cite[Remark~6.7]{tignol-wadsworth-2015}, $E$ has a 
 unique-up-to-isomorphism primary decomposition
 ${E \cong_g E_{p_1}\otimes _T \ldots \otimes_T E_{p_k}}$, where each $E_{p_j}$ is a graded sub-division 
 algebra of $E$ with 
 center $T$ and $\opn{ind}(E_{p_j}) = p_j^{r_j}$. 
 
 \begin{theorem} [Primary Decomposition] With $E$, $T$, and the $E_{p_j}$ as above, 
 \label{prdecompTK_1E}
 \begin{equation}
 \label{TK_1Eprdecomp}
        \mathrm{TK}_1(E)\,
        \cong\,
        \big(
        {\textstyle\bigoplus}_{j=1}^k \mathrm{TK}_{1,p_j}(E_{p_j})
        \big)
        \,{\textstyle\bigoplus}\,
        \big(
        \underset{p\notin\{p_1,\ldots,p_k\}}{\textstyle\bigoplus}
        \mu_{T_0}[p]
        \big)
 \end{equation}
 \end{theorem} 
 
\begin{proof}
 Let $K = q(T)$, $Q = q(E) = E \otimes_T K$, and $Q_j = q(E_{p_j})$ for $j = 1,2, \ldots, k$.  
 So, 
 $$
 Q_1 \otimes_K \ldots \otimes_K Q_k \,\cong\,  q(E_{p_1} \otimes_T \ldots \otimes_T E_{p_k})\, \cong \,q(E) \,=\,Q.
 $$
 Since $\opn{ind}(Q_j) = \opn{ind}(E_{p_j}) = p_j^{r_j}$ for each $j$ and 
 $\opn{ind}(Q) = \opn{ind} (E) = p_1^{r_1} \ldots p_k^{r_k}$,
 it follows that $Q_1, \ldots ,Q_k$ are the components in the primary decomposition of $Q$.  
  Note also that $\mu_ K = \mu_T$, as the commutative integral domain $T$  is integrally closed 
 (see \cite[Prop.~5.2]{tignol-wadsworth-2015}), and $\mu_T = \mu_{T_0}$.  
  By the stability theorem~\ref{thm:stability-sk1-tk1} for $\TK$, we have $\TK(E) \cong \TK(Q)$ and 
  $\TK(E_{p_j}) \cong \TK(Q_j)$, 
  so  $\opn{TK}_{1, p_j}(E_{p_j}) \cong \opn{TK}_{1, p_j}(Q_j)$, for each $j$.  With these isomorphisms,
  the isomorphism (\ref{TK_1Eprdecomp}) follows from (\ref{prdecompD}) applied to $Q$.
 \end{proof}

 \begin{corollary} Let $E_1$ and $E_2$ be graded division algebras with center  $T$ such that 
 ${\opn{gcd}(\opn{ind}(E_1), \opn{ind}(E_2)) = 1}$.  Then, 
 \begin{equation}
\label{TKEtensorprod}
       \TK(E_1\otimes_T E_2)\, \cong\,
        \Big(\underset{\,p\,\mid\,\opn{ind}(E_1)}{\textstyle\bigoplus}\opn{TK}_{1,p}(E_1)\Big)
        \,{\textstyle{\bigoplus}}\,
        \Big(\underset{\,p\,\mid\,\operatorname{ind}(E_2)}{\textstyle{\bigoplus}}\opn{TK}_{1,p}(E_2)\Big)
         \,{\textstyle{\bigoplus}}\,
        \Big(\underset{\,p\,\nmid\,\opn{ind}(E_1)\opn{ind}(E_2)}{\textstyle{\bigoplus}}       \mu_{T_0}[p]\Big).
\end{equation} 
 \end{corollary}
 
 \begin{proof}
 This follows immediately from Theorem~\ref{prdecompTK_1E}, just as (\ref{TKtensorprod}) above follows 
 from (\ref{prdecompD}).
 \end{proof}
 
 We give another application of the stability theorem for $\TK$ to prove a graded version of Motiee's 
 injectivity theorem \cite[Th.~7]{motiee-2013} for $\TK$ for suitable field extensions of the center.  For the 
 next propositions we consider a graded division algebra $E$ finite-dimensional over its center $T$ and 
 a finite-degree graded field extension $S$ of $T$.  Let $\Omega_T$ be the divisible hull of $\Gamma_T$
 (so $\Omega_T \cong \Gamma_T \otimes_\zz \qq$).   Then, as $\Gamma_E$ is torsion-free and 
 $\Gamma_E/\Gamma_T$ is a torsion group, there is a unique monomorphism $\Gamma_E \to \Omega_T$
 extending the inclusion $\Gamma_T \hookrightarrow \Omega_T$.  Therefore, we can view $\Gamma_E$
 as a subgroup of $\Omega_T$.  Likewise,  we can view 
 $\Gamma_S \subseteq \Omega_T$.  Thus, there is no ambiguity in forming the graded tensor product 
 $E \otimes_T S$ with grade group in $\Omega_T$.

 \begin{proposition}
\label{prop:prime-to-index-injectivity}
Let $E$ be a finite-dimensional graded division algebra over its center $T=Z(E)$. Let $T\subseteq S$ be a 
finite-degree graded field extension. If $\gcd(\operatorname{ind}(E),[S:T])=1$, then $E_S=E\otimes_T S$ is a 
graded division algebra, and the natural homomorphism
$$
        \rho_{S/T}:\mathrm{TK}_1(E)\longrightarrow \mathrm{TK}_1(E_S)
$$
is injective.
\end{proposition}

\begin{proof}
Put $K=q(T)$, $L=q(S)$, and $Q=q(E)$, which is a division algebra with center $K$.
 Since $S$ is a finite-degree  graded field extension of~$T$, so $S$ is 
integral over $T$, we have $q(S) = S\otimes_T q(T)$ (cf. \cite[Lemma~2.15]{tignol-wadsworth-2015}). That, is the 
quotient field 
$L$ of $S$ coincides with the quotient division algebra of $S$ as a graded $T$-algebra.  Hence also, $q(E_S)$ 
is the same (up to isomorphism) whether $E_S$ is viewed as an $S$-algebra or a $T$-algebra.  Note that
\begin{equation}
\label{q(E_S)}
Q\otimes_K L \,= \, q(E) \otimes_K q(S) \, = \, (E \otimes_T K)\otimes _K (S\otimes _T K) \,
\cong \, (E \otimes _T S) \otimes _T K \, = \, q(E_S).
\end{equation}
Moreover, 
$\opn{ind}(Q)=\opn{ind}(E)$ and $[L:K]=[S:T]$. Hence, 
\begin{equation}
\label{gcds}
\gcd(\opn{ind}(Q),[L:K]) \,=\,\gcd(\opn{ind}(E),[S:T])\,=\ 1.
\end{equation}
Therefore $Q\otimes _K L$ is a division ring. Since $q(E_S)$ is a division algebra, by 
\cite[Prop.~2.28]{tignol-wadsworth-2015} $E_S$ is a graded division algebra.

Now Motiee's injectivity theorem \cite[Th.~7]{motiee-2013} applies to the  central division algebra $Q$ over 
$K$ and the finite-degree field extension $L/K$. It says that the natural homomorphism
$$
        \rho_Q:\mathrm{TK}_1(Q)\longrightarrow \mathrm{TK}_1(Q\otimes_K L)
$$
is injective. Since $q(E_S)\cong Q\otimes_K L$, we regard this as an injective homomorphism
$
        \rho_Q:\mathrm{TK}_1(q(E))\longrightarrow \mathrm{TK}_1(q(E_S)).
$

By the stability theorem \ref{thm:stability-sk1-tk1} for $\mathrm{TK}_1$, the inclusions $E^*\hookrightarrow q(E)^*$ 
and $E_S^*\hookrightarrow q(E_S)^*$ induce isomorphisms
$
        \theta_E:\mathrm{TK}_1(E)\overset{\sim}{\longrightarrow}\mathrm{TK}_1(q(E))
 $
 and
 $       
        \theta_{E_S}:\mathrm{TK}_1(E_S)\overset{\sim}{\longrightarrow}\mathrm{TK}_1(q(E_S)).
$
Furthermore, it is easy to see that $\theta_{E_S}\circ \rho_{S/T}=\rho_Q\circ\theta_E$.
Hence, $\rho_{S/T}= \theta_{E_S}^{-1}\circ \rho_Q\circ\theta_E$, which is injective, 
as $\theta_{E_S}^{-1}$, $ \rho_Q$, and~$\theta_E$ are each injective.
\end{proof}

The preceding proposition gives injectivity of $\TK$ for a scalar extension
under the prime-to-index hypothesis.  If this hypothesis is dropped but $E_S$ is still a graded division algebra, 
we show that there is still some control over the kernel. The point is that the kernel is already contained 
in~$\opn{SK}_1(E)$.

\begin{proposition}
Let $E$ be a graded division algebra finite-dimensional over its center $T$,  let $S$ be a finite-degree
 graded field extension of $T$, and let $E_S = E \otimes_T S$.  Suppose that $E_S$ is a graded division
 algebra.  Let
$$
        \rho_{S/T}:\operatorname{TK}_1(E)\longrightarrow \operatorname{TK}_1(E_S)
$$
be the natural map.  Then,
\begin{enumerate}[font=\normalfont]
\item[(i)]
$\ker(\rho_{S/T}) =  \ker\bigl(\opn{SK}_1(E)\longrightarrow \opn{SK}_1(E_S)\bigr)
\subseteq \operatorname{SK}_1(E)$.
\item[(ii)]
We have 
$$
[S:T]\cdot\ker(\rho_{S/T})=0 \quad \text{and} \quad\operatorname{ind}(E)\cdot\ker(\rho_{S/T})=0.
$$
Consequently, the exponent of $\ker(\rho_{S/T})$ divides $\gcd(\operatorname{ind}(E),[S:T])$.  
\item[(iii)]
For every prime $p$ with $p\nmid \gcd(\operatorname{ind}(E),[S:T])$, the map
$$
        \operatorname{TK}_{1,p}(E)\longrightarrow \operatorname{TK}_{1,p}(E_S)
$$
is injective.
\end{enumerate}
\end{proposition}

\begin{proof}
For (i), 
Let $x\in\ker(\rho_{S/T})$, and write $x=aE'$, with $a\in E^{(\mu)}$.  Since $\rho_{S/T}(x)=1$, we have 
$a\otimes1\in E_S'$.  Hence $\operatorname{Nrd}_{E_S}(a\otimes1)=1$.  By the compatibility of reduced norms 
under scalar extension, $\operatorname{Nrd}_{E_S}(a\otimes1)=\operatorname{Nrd}_E(a)$, viewed as 
an element of $S^*$.  Since $T^*\hookrightarrow S^*$, it follows that $\operatorname{Nrd}_E(a)=1$.  Thus 
$a\in E^{(1)}$, and hence $x\in\operatorname{SK}_1(E)$.  Therefore, 
$\ker(\rho_{S/T})\subseteq \operatorname{SK}_1(E)$.

For the reverse inclusion, note that every element of $\operatorname{SK}_1(E)$ which maps trivially in 
$\operatorname{SK}_1(E_S)$ clearly maps trivially in $\operatorname{TK}_1(E_S)$.  Hence,
$
        \ker(\rho_{S/T})
        =
        \ker\bigl(\operatorname{SK}_1(E)\to \operatorname{SK}_1(E_S)\bigr).
$

For (ii),
Put $K=q(T)$, $L=q(S)$, and $Q=q(E)$.  Since $S$ is  a finite-degree  graded field extension of $T$
and $E_S$ is a graded division algebra, 
we saw in (\ref{q(E_S)}) that
$
        q(E_S)\cong Q\otimes_K L.
$
As we noted in the preceding proof, by the stability theorem  $\rho_{S/T}$ identifies with the  scalar extension map 
$\rho_{L/K}\colon\opn{TK}_1(Q)\to \opn{TK}_1(Q\otimes_K L)$.  For ordinary central simple algebras, restriction 
and corestriction on $K_1$ 
satisfy
$$
        \opn{cor}_{L/K}\circ\opn{res}_{L/K}
        \,=\,
        [L:K]\cdot\opn{id}_{K_1(Q)}.
$$
Therefore, every element of $\ker(\rho_{L/K})$ is killed by $[L:K]$.
Since $[L:K]=[S:T]$, every element in $\ker(\rho_{S/T})$ is therefore killed by $[S:T]$.  On the other hand, by the 
preceding 
proposition
$\ker(\rho_{S/T})\subseteq\opn{SK}_1(E)$.  Since $\operatorname{SK}_1(E)$ is killed by $\operatorname{ind}(E)$, 
the kernel is also killed by $\operatorname{ind}(E)$.  Hence its exponent divides $\gcd(\operatorname{ind}(E),[S:T])$. 

Part (iii) follows  immediately from part (ii). 
\end{proof}

\section*{Acknowledgements}
The first author (H. V. Khanh) was supported by the Vietnam National Foundation for Science and Technology Development (NAFOSTED) under Grant No.~101.04-2025.41.


\begin{thebibliography}{999999}

\bibitem[D]{draxl-1983} P. Draxl, {\it{Skew Fields}}, London Math.~Soc.~Lec.~Note Ser., Vol.~81, 
Cambridge Univ.~Press, Cambridge, 1983.

\bibitem[E]{ershov-1982} Yu.~L.~Ershov, Henselian valuations of division rings and the group $\SK$,
Mat.~Sb.~(N.S.), \textbf{117} (159),  (1982), no.~1, 60--68 (in Russian); English transl: Math.~USSR-Sb., 
\textbf{45} (1983), 63--71.


\bibitem[HaW1]{hazrat-wadsworth-2011}
R.~Hazrat and A.~R.~Wadsworth, $\mathrm{SK}_1$ of graded division algebras, Israel J.~Math., \textbf{183} (2011), 117--163.

\bibitem[HaW2]{hazrat-wadsworth-unitary-2011}
R.~Hazrat and A.~R.~Wadsworth, 
Unitary $\mathrm{SK}_1$ of graded and valued division algebras, 
Proc.~London Math.~Soc. (3), \textbf{103} (2011), no.~3, 508--534.

\bibitem[HwW]{hwang-wadsworth-99}
Y.-S.~Hwang and A.~R.~Wadsworth, Correspondences between valued division algebras and graded division algebras, J. Algebra, \textbf{220} (1999), 73--114.

\bibitem[KK]{khanh-khoa-2025}
H.~V.~Khanh and N.~D.~A.~Khoa,
On the congruence theorem for valued division algebras,
Arch.~Math., \textbf{125} (2025), 369--377.

\bibitem[J]{jacobson-1996}
N.~Jacobson, {\it{Finite-Dimensional Division Algebras over Fields}}, Springer, Berlin, 1996.

\bibitem[M]{motiee-2013}
M.~Motiee,
On torsion subgroups of Whitehead groups of division algebras,
Manuscripta Math., \textbf{141} (2013), no.~3--4, 717--726.


\bibitem[P1]{platonov-1976}
V.~P.~ Platonov, The Tannaka-Artin problem, and reduced K-theory,   Izv.~Akad.~Nauk SSSR Ser.~Mat.,
\textbf{40} (1976), no.~2, 227--261, 469 (in Russian); English transl.:  Math.~USSR Izv., 
\textbf{10} (1976), 211--243.

\bibitem[P2]{platonov-stability-1976}
V.~P.~Platonov, Reduced K-theory and approximation in algebraic groups, pp.~198--207, 270, in 
Number theory, mathematical analysis and their applications,
Trudy Mat. Inst. Steklov. \textbf{142} (1976) (in Russian); English transl.: Proc.~Steklov Inst.~Math., 
\textbf{142} (1979), 213--224. 


\bibitem[PY]{platonov-yanchevskii-1979} V.~P.~Platonov and V.I. Yanchevski\v{\i}, $\opn{SK}_1$ for division rings of 
noncommutative rational functions,
Dokl.~Akad.~Nauk SSSR, \textbf{249} (1979), no.~5, 1064--1068 (in Russian); English transl.: Soviet Math.~Dokl.,
\textbf{20} (1979), no.~6, 1393--1397 (1980).

\bibitem[TW]{tignol-wadsworth-2015}
J.-P.~Tignol and A.~R.~Wadsworth,
\emph{Value Functions on Simple Algebras, and Associated Graded Rings},
Springer Monographs in Mathematics, Springer, Cham, 2015.


\end{thebibliography}
\end{document}